\documentclass[11pt, reqno]{amsart}

\usepackage{amsmath, amsthm, amscd, amsfonts, amssymb, graphicx, color, mathtools, mathrsfs}
\usepackage[bookmarksnumbered, colorlinks, plainpages]{hyperref}
\usepackage{enumerate}
\usepackage{setspace}
\usepackage{multicol}
\usepackage[margin=1in]{geometry}
\usepackage{comment}
\usepackage{booktabs}
\usepackage{caption}
\usepackage{subcaption}
\usepackage{pgfplots}
\usepackage{algorithm}

\newcounter{lititem}[subsection]

\theoremstyle{definition}
\newtheorem{theorem}{Theorem}[section]
\newtheorem{lemma}[theorem]{Lemma}
\newtheorem{proposition}[theorem]{Proposition}
\newtheorem{assumption}[theorem]{Assumption}
\newtheorem{corollary}[theorem]{Corollary}

\numberwithin{equation}{section}

\makeatletter
\let\c@algorithm\c@theorem

\@ifundefined{theHtheorem}{}{}
\makeatother

\newcommand{\bff}{\boldsymbol}
\newcommand{\bb}{\mathbb}

\newcommand{\dbeta}{\mathrm{d}\beta}
\newcommand{\dt}{\mathrm{d}t}

\newcommand{\dx}{\mathrm{d}x}
\newcommand{\ds}{\mathrm{d}s}

\newcommand{\dW}{\mathrm{d}W}

\newcommand{\norm}[2]{\left\|{#1}\right\|_{#2}}
\newcommand{\inpro}[2]{\left\langle#1,#2\right\rangle}

\allowdisplaybreaks

\begin{document}
	\setcounter{page}{1}
	
	\title[Rate of convergence of a midpoint scheme for the sLLG equation]
	{Rate of convergence of a fully discrete structure-preserving midpoint scheme for the stochastic Landau--Lifshitz--Gilbert equation}
	
	\author[Agus L. Soenjaya]{Agus L. Soenjaya}
	\address{Institute of Analysis and Scientific Computing, TU Wien, Wiedner Hauptstrasse 8--10, 1040 Vienna, Austria}
	\email{\textcolor[rgb]{0.00,0.00,0.84}{agus.soenjaya@asc.tuwien.ac.at}}
	
	\thanks{\textbf{Acknowledgment.} This research was initiated while the author was supported by the Commonwealth through an Australian Government Research Training Program (RTP) Scholarship (\href{https://doi.org/10.82133/C42F-K220}{DOI: 10.82133/C42F-K220}) at UNSW Sydney. The author has also been supported by an Australian Mathematical Society Lift-off Fellowship. The work was completed during the author's postdoctoral appointment at the Institute of Analysis and Scientific Computing of TU Wien, supported by the Austrian Science Fund (FWF) through the international project I6802 ``Functional error estimates for PDEs on unbounded domains'' (\href{https://doi.org/10.55776/I6802}{DOI: 10.55776/I6802}), led by Prof. Dirk Praetorius.}
	
	\begin{abstract}
		The stochastic Landau--Lifshitz--Gilbert (sLLG) equation is a strongly nonlinear stochastic PDE with a non-convex pointwise constraint arising in the theory of micromagnetics. We analyse a fully discrete, structure-preserving finite element approximation of the sLLG equation with coloured multiplicative Stratonovich noise on a bounded interval. The method utilises continuous piecewise affine finite elements, mass lumping, and midpoint time discretisation to preserve the unit-length constraint exactly at the finite element nodes. Under suitable regularity assumptions on the initial data and the noise, we establish uniform higher-moment stability and develop an error analysis for the scheme. The analysis exploits the geometric structure of the equation and the stochastic midpoint discretisation. For every $\gamma\in(0,\frac12)$, we prove first-order spatial convergence and temporal convergence of order $\gamma$ in the natural discrete energy norm, locally in mean square on events of arbitrarily large probability and,
		consequently, in probability. To the best of our knowledge, this is the first convergence-rate result for a fully discrete structure-preserving finite element scheme solving the stochastic Landau--Lifshitz--Gilbert equation.
	\end{abstract}
	\maketitle
	
\section{Introduction}

The Landau--Lifshitz--Gilbert (LLG) equation is a fundamental model for the
dynamics of the magnetisation in ferromagnetic materials; see, e.g.,~\cite{BanBrzNekPro13, GuoDing08, Lak11, Pro01} and the references therein. At nonzero temperature, thermal fluctuations may be incorporated through a stochastic magnetic field, leading to the stochastic Landau--Lifshitz--Gilbert (sLLG) equation. This gives a strongly nonlinear SPDE with a non-convex pointwise constraint, in which the multiplicative noise interacts nontrivially with the geometric cross-product structure. The noise is naturally interpreted in the Stratonovich sense, and is compatible with the pointwise constraint on the magnetisation since it acts tangentially to the unit sphere.

In this work, we consider an exchange-only sLLG model on the bounded interval
$D=(0,L)$. Let $T>0$ and let $(\Omega,\mathcal F,\{\mathcal F_t\}_{t\in[0,T]},\mathbb P)$
be a filtered probability space satisfying the usual conditions. For
$\alpha,\kappa>0$, the sLLG equation reads:
\begin{subequations}\label{eq:sllg-system}
	\begin{alignat}{2}
		&\mathrm d\bff m
		+
		\alpha\kappa
		\bff m\times
		\left(
		\bff m\times\partial_{xx}\bff m
		\right)\dt
		-
		\kappa
		\bff m\times\partial_{xx}\bff m\,\dt
		=
		\bff m\times\circ\,\dW
		&&\quad\text{in } (0,T)\times D,
		\label{eq:sllg-model}
		\\
		&\partial_x\bff m
		=\bff0
		&&\quad\text{on } (0,T)\times \partial D,
		\label{eq:sllg-neumann-boundary}
		\\
		&\bff m(0)
		=\bff m_0,
		&&\quad\text{in }D,
		\label{eq:sllg-initial-condition}
	\end{alignat}
\end{subequations}
where $\bff{m}:[0,T]\times D\times \Omega\to \bb{R}^3$ is the magnetisation.
Here, we assume $\bff m_0\in H^3(D;\mathbb S^2)$. In \eqref{eq:sllg-model}, the first two terms in the drift describe damping and precession, respectively, while the Stratonovich term models thermal fluctuations. Lower-order anisotropy field can also be added to \eqref{eq:sllg-model} without difficulties. We
consider the simplified multiplicative noise $\bff m\times\circ\,\dW$ in the sense of Stratonovich, as in~\cite{BanBrzPro13, BrzGolJeg12, GusHoc23}, and assume sufficient spatial regularity of the coloured Wiener process $W$ throughout the analysis. An important property of \eqref{eq:sllg-system} is that the magnetisation satisfies a non-convex geometric constraint
\begin{equation}\label{eq:exact-unit-length-sllg}
	|\bff m(t,x)|=1
	\qquad
	\text{for a.e. }(t,x)\in(0,T)\times D,
	\quad\mathbb P\text{-a.s.}
\end{equation}
Our analysis also applies to the one-dimensional flat torus with periodic boundary condition.

The restriction to one space dimension is motivated by both the analytical
theory and applications. From the analytical viewpoint, the global
pathwise well-posedness and maximal regularity required for a quantitative
error analysis in the present paper are available for the sLLG equation on an
interval; see, e.g., \cite{Dun15, GusHoc23}. In higher dimensions, the
analytical theory is substantially more delicate, owing to the possibility of singularity formation and the weaker regularity available for solutions. The one-dimensional setting is also a
natural model for thin ferromagnetic wires, and has been considered in a number of mathematical studies of stochastic
magnetisation dynamics~\cite{BanBrzPro13, BrzHauLi19}.

The deterministic LLG equation has been studied extensively from both
analytical and numerical viewpoints; see, e.g.,
\cite{AkrFeiKovLub21, AloSoy92, Cim08, FeiTra17, GuoDing08, Pro01} and the references therein. In particular, the midpoint scheme for the deterministic problem was first analysed in~\cite{BarPro06}. The stochastic problem is substantially more delicate. In addition to the nonlinearity and non-convex sphere constraint, the analysis must accommodate multiplicative
Stratonovich noise, limited solution regularity, and the interaction
between the stochastic forcing and geometric structure of the
equation~\cite{BanBrzNekPro13, BrzGolJeg12, BrzGolJeg17, GusHoc23}. Three types of convergent numerical methods for sLLG have been
developed: 
\begin{enumerate}[(i)]
	\item midpoint scheme introduced in~\cite{AquSerCopMay06, BanBrzPro13}, which preserves the unit-length constraint at nodal points and converges to a weak martingale solution along a subsequence~\cite{BanBrzNekPro14};
	\item tangent plane scheme~\cite{AloBouHoc14} on the Gilbert form, which extended the scheme in~\cite{Alo08} to the stochastic case;
	\item tangent plane scheme~\cite{GolLeTra16} on a random PDE which is obtained by means of the Doss--Sussmann transform.
\end{enumerate}
For the one-dimensional problem, Dunst~\cite{Dun15} subsequently proved convergence in probability with rate arbitrarily close to $1/2$ for the \emph{time-semidiscrete} midpoint scheme.

Extending such a rate estimate to a \emph{fully discrete} method is not immediate: the finite element approximation introduces mass-lumping and interpolation defects, while the discrete Laplacian and the stochastic midpoint term require estimates which have no direct counterpart in the time-semidiscrete
analysis. Indeed, the lack of suitable uniform control of the discrete
Laplacian was already identified in \cite{Dun15} as a main obstruction to
a fully discrete rate analysis. To the best of our knowledge, no
convergence rate has previously been proved for a fully discrete
structure-preserving approximation of the sLLG equation.
Even for deterministic LLG, rigorous optimal convergence rates towards strong
solutions for the corresponding midpoint scheme have
only been obtained recently~\cite{Soe26llg}. This highlights the difficulty
of quantitative error analysis for the midpoint scheme, even before additional complications caused by multiplicative Stratonovich noise are introduced.

The aim of this work is to close this gap for the one-dimensional sLLG
equation with coloured multiplicative noise. We analyse a mass-lumped
$P_1$ finite element discretisation combined with the midpoint rule in
time, which preserves the unit-length constraint exactly at every
finite-element node. The main difficulty is to obtain quantitative error
bounds without destroying this geometric structure. To this end, we first
derive uniform higher-moment estimates for the discrete exchange energy
and time increments. We then develop a localised error analysis in which
the deterministic drift consistency error and the stochastic midpoint
defect are treated separately, since the latter requires a delicate control of diagonal and off-diagonal quadratic Brownian increments. Under suitable initial data and noise regularity assumptions, this yields a fully discrete convergence rate in probability of order $h+k^\gamma$ for every $\gamma<\frac12$ in the energy norm. In particular, this provides the first rigorous convergence rate result for a fully discrete structure-preserving scheme for sLLG.

Our main result shows that, under suitable regularity assumptions on the
initial data and the noise, for every $\gamma\in(0,\frac12)$ and every
fixed localisation level $\Lambda>0$,
\[
\mathbb E\left[
\mathbf1_{\Omega_{\Lambda,J}}
\max_{0\le j\le J}
\norm{\bff m(t_j)-\bff m_h^j}{\bb L^2}^2
\right]
+
\mathbb E\left[
\mathbf1_{\Omega_{\Lambda,J}}
k\sum_{j=0}^{J-1}
\norm{
	\partial_x
	\bigl(
	\bff m(t_{j+1})-\bff m_h^{j+1}
	\bigr)
}{\bb L^2}^2
\right]
\le
C_{\Lambda,\gamma}
\left(
h^2+k^{2\gamma}
\right).
\]
Together with the probability estimate for the localisation sets, this
implies a convergence rate of order $h+k^\gamma$, for every
$\gamma<\frac12$, in probability in the discrete
$L^\infty(0,T;\bb L^2)\cap L^2(0,T;\bb H^1)$ energy norm. Thus, the
spatial discretisation does not deteriorate the temporal rate obtained
for the time-semidiscrete midpoint method in \cite{Dun15}, while yielding
simultaneously an optimal spatial rate in the $\bb{H}^1$-norm.
Since comprehensive numerical experiments for midpoint discretisation of sLLG are already available in \cite{BanBrzPro13, Dun15}, we do not repeat such simulations here and focus instead on the rigorous fully discrete convergence-rate analysis.

The remainder of the paper is organised as follows.
Section~\ref{sec:prelim} introduces the notation, standing assumptions, and
finite-element preliminaries. In Section~\ref{sec:stability-error}, we define
the fully discrete midpoint scheme, establish its nodal constraint
preservation and moment stability, and derive the consistency and localised
error estimates. Various technical discrete estimates used in the analysis are collected in the appendix.

\section{Preliminaries}\label{sec:prelim}

In this section, we collect the notation, assumption, and finite element tools used throughout the paper.

\subsection{Notation}\label{subsec:notation}

For $p\in[1,\infty]$ and $s\ge0$, set
\[
\bb L^p:=L^p(D;\bb R^3),
\qquad
\bb W^{s,p}:=W^{s,p}(D;\bb R^3),
\qquad
\bb H^s:=\bb W^{s,2}.
\]
For a Banach space $X$, the spaces $L^p(0,T;X)$,
$W^{s,p}(0,T;X)$, and $C^\gamma([0,T];X)$ have their usual
vector-valued meanings. The operators $\partial_x$ and
$\partial_{xx}$ act componentwise.

The scalar product in a Hilbert space $H$ is denoted by
$\inpro{\cdot}{\cdot}_H$, with corresponding norm $\norm{\cdot}{H}$.
When $H=\bb L^2$, we simply write
\[
\inpro{\bff u}{\bff v}
:=
\int_D \bff u\cdot\bff v\,\dx .
\]
We use the same notation for vector- and matrix-valued functions whenever
the meaning is clear from the context.

Throughout, $C>0$ denotes a generic constant that may depend on fixed
problem data but is independent of the discretisation parameters $h$ and
$k$. Dependence on specific parameters is indicated when relevant.

\subsection{Assumption}\label{subsec:assum}

The following assumption on initial data and noise regularity is made throughout this paper.

\begin{assumption}\label{ass:noise-regularity-sllg}
	We assume that the initial data $\bff m_0\in H^3(D;\mathbb S^2)$ with $\partial_x \bff{m}_0=\bff{0}$. To describe the noise, let
	\begin{equation}\label{eq:noise-expansion-assumption-sllg}
		W(t)=\sum_{\ell=1}^{\infty}\bff g_\ell\beta_\ell(t),
	\end{equation}
	where $\{\beta_\ell\}_{\ell =1}^\infty$ are independent standard
	$\{\mathcal F_t\}$-Brownian motions and
	$\bff g_\ell\in \bb{H}^3$ are
	deterministic, such that
	\begin{equation}\label{eq:noise-W1inf-l1-assumption-sllg}
		\Gamma
		:=
		\sum_{\ell=1}^{\infty}
		\norm{\bff g_\ell}{\bb H^3}
		<\infty.
	\end{equation}
	We assume the compatibility condition $\partial_x\bff g_\ell=\bff 0$ on $\partial D$ for all $\ell\ge 1$.
	This boundary compatibility condition is only needed to ensure that the
	stochastic vector fields preserve the homogeneous Neumann condition of a sufficiently regular solution.
\end{assumption}

Assumption~\ref{ass:noise-regularity-sllg} guarantees the existence of a strong pathwise solution to the stochastic LLG equation~\eqref{eq:sllg-system} with the following regularity:
for every $q\ge1$ and every $\gamma\in(0,\frac12)$,
\begin{equation}\label{eq:exact-solution-regularity-sllg}
	\mathbb E\left[
	\norm{\bff m}{L^\infty(0,T;\bb H^3)}^{2q}
	+
	\norm{\bff m}{C^\gamma([0,T];\bb H^2)}^{2q}
	\right]
	\le
	C_{q,\gamma}.
\end{equation}
Such pathwise well-posedness and regularity are established in~\cite{Dun15} for the periodic problem; also see~\cite{GusHoc23}. The corresponding Neumann-boundary result follows by the same argument under the compatibility conditions imposed above.

\subsection{Finite element approximation}\label{subsec:fe-prelim-stochastic}

Let $\{\mathcal T_h\}_{h>0}$ be a family of shape-regular and quasi-uniform
partitions of $D$ into intervals, with maximal mesh-size $h>0$. We denote by
$\mathcal N_h$ the set of nodes of $\mathcal T_h$. We use the lowest-order
conforming Lagrange space
\begin{equation}\label{eq:Vh-def-sllg}
	\bb V_h
	:=
	\left\{
	\bff\phi_h\in C^0(\overline D;\bb R^3)
	:
	\bff\phi_h|_K\in\mathcal P_1(K;\bb R^3),
	\;\forall K\in\mathcal T_h
	\right\}
	\subset\bb H^1.
\end{equation}
We equip $\bb V_h$ with the norm
\begin{equation}\label{eq:discrete-H1-norm-sllg}
	\norm{\bff\phi_h}{1,h}
	:=
	\norm{\bff\phi_h}{h}
	+
	\norm{\partial_x\bff\phi_h}{\bb L^2}.
\end{equation}
For a functional $\mathcal R\in\bb V_h'$, we define the corresponding dual
norm by
\begin{equation}\label{eq:dual-Hminus1h-def-sllg}
	\norm{\mathcal R}{-1,h}
	:=
	\sup_{\bff0\ne\bff\phi_h\in\bb V_h}
	\frac{
		|\mathcal R(\bff\phi_h)|
	}{
		\norm{\bff\phi_h}{1,h}
	}.
\end{equation}

Let $\{\varphi_z\}_{z\in\mathcal N_h}$ be the scalar nodal basis and denote
by $I_h$ the nodal interpolation operator, for both scalar- and
vector-valued functions. Since $D\subset\bb R$ is a bounded interval,
$\bb H^1(D)\hookrightarrow C^0(\overline D)$. For continuous vector fields
$\bff u,\bff v\in C^0(\overline D;\bb R^3)$, define the mass-lumped inner
product by
\begin{equation}\label{eq:mass-lumped-inner-product-sllg}
	\inpro{\bff u}{\bff v}_h
	:=
	\int_D I_h(\bff u\cdot\bff v)\,\dx
	=
	\sum_{z\in\mathcal N_h}
	\beta_z\,\bff u(z)\cdot\bff v(z),
	\qquad
	\beta_z:=\int_D\varphi_z\,\dx .
\end{equation}
We write
\begin{equation}\label{eq:lumped-norm-def-sllg}
	\norm{\bff v_h}{h}^2
	:=
	\inpro{\bff v_h}{\bff v_h}_h .
\end{equation}
The norm $\norm{\cdot}{h}$ is uniformly equivalent to the standard
$\bb L^2$-norm on $\bb V_h$, namely
\begin{equation}\label{eq:lumped-L2-equivalence-sllg}
	c\norm{\bff v_h}{\bb L^2}^2
	\le
	\norm{\bff v_h}{h}^2
	\le
	C\norm{\bff v_h}{\bb L^2}^2,
	\qquad
	\forall\bff v_h\in\bb V_h,
\end{equation}
where $c,C>0$ are independent of $h$. The following discrete Agmon
inequality is a consequence of \eqref{eq:lumped-L2-equivalence-sllg} and
the standard one-dimensional Agmon inequality:
\begin{equation}\label{eq:discrete-agmon-sllg}
	\norm{\bff v_h}{\bb L^\infty}^2
	\le
	C\norm{\bff v_h}{h}\norm{\partial_x\bff v_h}{\bb L^2}
	+
	C\norm{\bff v_h}{h}^2,
	\qquad
	\forall\bff v_h\in\bb V_h.
\end{equation}

We further define the mass-lumped discrete Laplacian
$\Delta_h:\bb V_h\to\bb V_h$ by
\begin{equation}\label{eq:def-discrete-laplacian-sllg}
	\inpro{\Delta_h\bff v_h}{\bff\chi_h}_h
	=
	-\inpro{\partial_x\bff v_h}{\partial_x\bff\chi_h},
	\qquad
	\forall\bff v_h,\bff\chi_h\in\bb V_h.
\end{equation}

We shall use the following standard approximation and inverse estimates.
For each $s\in[0,1]$, there exists $C>0$ such that, for all
$\bff v\in\bb H^{1+s}$,
\begin{equation}\label{eq:Ih-H1-approx-sllg}
	\norm{\bff v-I_h\bff v}{\bb L^2}
	+
	h\norm{\partial_x(\bff v-I_h\bff v)}{\bb L^2}
	\le
	Ch^{1+s}\norm{\bff v}{\bb H^{1+s}}.
\end{equation}
In one-dimensional domain, we have the $\bb{H}^1$-stability estimate:
\begin{equation}\label{eq:Ih-stability}
	\norm{I_h\bff v}{\bb H^1}
	\le
	C\norm{\bff v}{\bb H^1}.
\end{equation}
Moreover, for every $\bff v_h\in\bb V_h$, $s\in\{0,1\}$, and
$1\le q\le p\le\infty$, the inverse estimate
\begin{equation}\label{eq:inverse-estimate-sllg}
	\norm{\bff v_h}{\bb W^{s,p}}
	\le
	Ch^{-\left(\frac1q-\frac1p\right)}
	\norm{\bff v_h}{\bb W^{s,q}}
\end{equation}
holds with $C$ independent of $h$.

In $\bb V_h$, we have the following quadrature estimates associated with
the mass-lumped inner product~\cite{BarPro06}:
\begin{alignat}{2}
	\label{eq:lumped-quadrature-Hminus1-sllg}
	\left|
	\inpro{\bff v_h}{\bff\chi_h}_h
	-
	\inpro{\bff v_h}{\bff\chi_h}
	\right|
	&\le
	Ch\norm{\bff v_h}{\bb L^2}
	\norm{\bff\chi_h}{\bb H^1},
	\qquad
	&&\forall\bff v_h,\bff\chi_h\in\bb V_h,
	\\
	\label{eq:smooth-lumped-quadrature-sllg}
	\left|
	\inpro{I_h\bff f}{\bff\chi_h}_h
	-
	\inpro{\bff f}{\bff\chi_h}
	\right|
	&\le
	Ch\norm{\bff f}{\bb H^1}
	\norm{\bff\chi_h}{1,h},
	\qquad
	&&\forall\bff f\in\bb H^1,\ \bff\chi_h\in\bb V_h.
\end{alignat}
We shall also repeatedly use the following elementary product estimates~\cite{Soe26llg}:
for all $\bff v_h,\bff w_h\in\bb V_h$,
\begin{align}
	\norm{I_h(\bff v_h\times\bff w_h)}{\bb L^2}
	&\le
	C\norm{\bff v_h}{\bb L^\infty}
	\norm{\bff w_h}{\bb L^2},
	\label{eq:Ih-product-L2-sllg}
	\\
	\norm{\partial_xI_h(\bff v_h\times\bff w_h)}{\bb L^2}
	&\le
	C\left(
	\norm{\bff v_h}{\bb L^\infty}
	\norm{\partial_x\bff w_h}{\bb L^2}
	+
	\norm{\bff w_h}{\bb L^\infty}
	\norm{\partial_x\bff v_h}{\bb L^2}
	\right),
	\label{eq:Ih-product-H1-sllg}
	\\
	\norm{\partial_xI_h(\bff v_h\times\bff w_h)}{\bb L^2}
	&\le
	C\left(
	\norm{\bff w_h}{\bb L^\infty}
	\norm{\partial_x\bff v_h}{\bb L^2}
	+
	\norm{\partial_x\bff w_h}{\bb L^\infty}
	\norm{\bff v_h}{h}
	\right).
	\label{eq:Ih-product-H1-L2refined-sllg}
\end{align}
Furthermore, noting~\eqref{eq:def-discrete-laplacian-sllg}, we have the following one-dimensional discrete cancellation estimate:
\begin{equation}\label{eq:discrete-cross-cancellation-sllg}
	\left|
	\inpro{\bff v_h\times\bff\eta_h}{-\Delta_h\bff v_h}_h
	\right|
	=
	\left|
	\inpro{
		\partial_xI_h(\bff v_h\times\bff\eta_h)
	}{
		\partial_x\bff v_h
	}
	\right|
	\le
	C
	\norm{\bff v_h}{\bb L^\infty}
	\norm{\partial_x\bff\eta_h}{\bb L^2}
	\norm{\partial_x\bff v_h}{\bb L^2}.
\end{equation}

We shall also need the following product-interpolation estimates.

\begin{lemma}[Product-interpolation estimates]
	\label{lem:product-interpolation-sllg}
	Let $\bff a\in\bb H^2$ and $\bff\phi_h\in\bb V_h$. Then
	\begin{align}
		\norm{
			\partial_x\left[
			I_h(\bff\phi_h\times I_h\bff a)
			-\bff\phi_h\times\bff a
			\right]
		}{\bb L^2}
		&\le
		Ch\norm{\bff a}{\bb H^2}\norm{\bff\phi_h}{1,h},
		\label{eq:product-interpolation-single-sllg}
		\\
		\norm{
			\partial_x\left[
			I_h\left(
			(\bff\phi_h\times I_h\bff a)\times I_h\bff a
			\right)
			-
			(\bff\phi_h\times\bff a)\times\bff a
			\right]
		}{\bb L^2}
		&\le
		Ch\norm{\bff a}{\bb H^2}^2\norm{\bff\phi_h}{1,h}.
		\label{eq:product-interpolation-double-sllg}
	\end{align}
\end{lemma}

\begin{proof}
	For every $z\in\mathcal N_h$, we have
	$I_h\bff a(z)=\bff a(z)$. Hence, the two differences are the
	elementwise interpolation errors of
	$\bff\phi_h\times\bff a$ and
	$(\bff\phi_h\times\bff a)\times\bff a$, respectively.
	The elementwise $\bb H^1$-interpolation estimate, the fact that
	$\bff\phi_h$ is affine on each element, and the one-dimensional
	embeddings $\bb H^1\hookrightarrow\bb L^\infty$ and $\bb H^2\hookrightarrow\bb W^{1,\infty}$
	yield the asserted bounds.
\end{proof}

Next, we record a moment estimate for the discrete Brownian increments.
Set
\begin{align}\label{eq:gh Wh}
	\bff g_{\ell,h}:=I_h\bff g_\ell,
	\qquad
	W_h(t):=\sum_{\ell=1}^{\infty}\bff g_{\ell,h}\beta_\ell(t).
\end{align}
For $j=0,\ldots,J-1$, define
\[
\hat{\Delta}_j\beta_\ell
:=
\beta_\ell(t_{j+1})-\beta_\ell(t_j),
\qquad
\hat{\Delta}_jW_h
:=
W_h(t_{j+1})-W_h(t_j)
=
\sum_{\ell=1}^{\infty}
\bff g_{\ell,h}\hat{\Delta}_j\beta_\ell.
\]
Under Assumption~\ref{ass:noise-regularity-sllg}, for every $q\ge1$,
\begin{equation}\label{eq:noise-increment-moment-sllg}
	\mathbb E\left[
	\norm{\hat{\Delta}_jW_h}{\bb W^{1,\infty}}^{2q}
	\right]
	\le
	C_q\Gamma^{2q}k^q.
\end{equation}
Indeed, the $\bb W^{1,\infty}$-stability of $I_h$ and Minkowski's
inequality give
\begin{align*}
	\left(
	\mathbb E\left[
	\norm{\hat{\Delta}_jW_h}{\bb W^{1,\infty}}^{2q}
	\right]
	\right)^{\frac{1}{2q}}
	&\le
	\sum_{\ell=1}^{\infty}
	\norm{I_h\bff g_\ell}{\bb W^{1,\infty}}
	\left(
	\mathbb E|\hat{\Delta}_j\beta_\ell|^{2q}
	\right)^{\frac{1}{2q}}
	\le
	C_q\Gamma k^{\frac12}.
\end{align*}

\section{Stability and error estimates of the fully discrete midpoint scheme}\label{sec:stability-error}

Let $0=t_0<t_1<\cdots<t_J=T$ be the uniform partition of $[0,T]$
with time-step $k=T/J$.
The fully discrete mass-lumped midpoint approximation is defined by the
following algorithm.

\begin{algorithm}
	\raggedright
	\caption{Mass-lumped midpoint scheme}
	\label{alg:midpoint-sllg}
	
	\textbf{Initialization.}
	Set $\bff m_h^0:=I_h\bff m_0$.
	
	\textbf{Time-stepping.}
	For $j=0,\ldots,J-1$, given $\bff m_h^j\in\bb V_h$, find
	$\bff m_h^{j+1}\in\bb V_h$ such that, with
	\begin{equation}\label{eq:midpoint-def-sllg}
		\bff m_h^{j+\frac12}
		:=
		\frac12(\bff m_h^{j+1}+\bff m_h^j),
	\end{equation}
	there holds, for every $\bff\phi_h\in\bb V_h$,
	\begin{align}
		&\inpro{\bff m_h^{j+1}-\bff m_h^j}{\bff\phi_h}_h
		+
		\alpha\kappa k
		\inpro{
			\bff m_h^{j+\frac12}
			\times
			\left[
			\bff m_h^{j+\frac12}
			\times\Delta_h\bff m_h^{j+1}
			\right]
		}{\bff\phi_h}_h
		\nonumber\\
		&\qquad-
		\kappa k
		\inpro{
			\bff m_h^{j+\frac12}
			\times\Delta_h\bff m_h^{j+1}
		}{\bff\phi_h}_h
		=
		\inpro{
			\bff m_h^{j+\frac12}
			\times\hat{\Delta}_jW_h
		}{\bff\phi_h}_h.
		\label{eq:fully-discrete-midpoint-sllg-laplacian}
	\end{align}
\end{algorithm}

Existence of an adapted solution $\{\bff m_h^j\}_{j=0}^{J}$ to Algorithm~\ref{alg:midpoint-sllg} was shown in~\cite{BanBrzNekPro14}.
In the subsequent analysis, we assume $k\in (0,1]$, which entails no restriction in the asymptotic error analysis.

\subsection{Stability estimates}\label{subsec:stability-estimates}

We first record the nodal constraint and the corresponding $\bb{L}^\infty$-bound.

\begin{lemma}[Nodal constraint and $\bb{L}^\infty$-bound]\label{lem:nodal-constraint-sllg}
Assume that $|\bff m_h^0(z)|=1$ for all $z\in\mathcal N_h$.
Then, for every $j=0,1,\ldots,J$, every solution of
\eqref{eq:fully-discrete-midpoint-sllg-laplacian} satisfies
\begin{equation}\label{eq:nodal-constraint-sllg}
    |\bff m_h^j(z)|=1,
    \qquad
    \forall z\in\mathcal N_h,
    \quad\mathbb P\text{-a.s.}
\end{equation}
Consequently, for every $j=0,1,\ldots,J$,
\begin{equation}\label{eq:Linfty-stability-sllg}
    \norm{\bff m_h^j}{\bb L^\infty}\le 1,
    \qquad
    \mathbb P\text{-a.s.}
\end{equation}
\end{lemma}

\begin{proof}
Fix a time level $j$ and a node $z\in\mathcal N_h$. Choose
$\bff\phi_h=\varphi_z\bff m_h^{j+\frac12}(z)$ in the discrete scheme.
Since the mass-lumped inner product is diagonal and $\beta_z>0$, all
cross-product terms vanish and
\[
\begin{aligned}
    0
    &=
    \beta_z
    \bigl(\bff m_h^{j+1}(z)-\bff m_h^j(z)\bigr)
    \cdot\bff m_h^{j+\frac12}(z)
    =
    \frac{\beta_z}{2}
    \left(
      |\bff m_h^{j+1}(z)|^2-|\bff m_h^j(z)|^2
    \right).
\end{aligned}
\]
This implies~\eqref{eq:nodal-constraint-sllg}. On every element,
$\bff m_h^j$ is a convex combination of its two nodal values. Hence,
$|\bff m_h^j(x)|\le1$ for all $x\in\overline D$, which proves
\eqref{eq:Linfty-stability-sllg}.
\end{proof}

We start with the following estimate for the stochastic term.

\begin{lemma} \label{lem:BDG-stochastic-energy-input}
Let $p\ge2$ and let $\{\bff m_h^j\}_{j=0}^{J}$ be an adapted solution of \eqref{eq:fully-discrete-midpoint-sllg-laplacian}.
Then
\begin{align}
    &\mathbb E\left[
      \max_{0\le r\le J}
      \left|
        \sum_{j=0}^{r-1}
        \kappa
        \inpro{
          \partial_xI_h(\bff m_h^j\times\hat{\Delta}_jW_h)
        }{\partial_x\bff m_h^j}
      \right|^p
    \right]
    \le
    C_{p,T,\kappa,\Gamma}
    \left(
      1+
      \sum_{j=0}^{J-1}
      k\,\mathbb E\left[
        \norm{\partial_x\bff m_h^j}{\bb L^2}^{2p}
      \right]
    \right),
    \label{eq:BDG-stochastic-energy-input-final}
\end{align}
where $C_{p,T,\kappa,\Gamma}$ is independent of $h$ and $k$.
\end{lemma}

\begin{proof}
Set $\mathfrak Z_{j+1}:=
    \kappa
    \inpro{
      \partial_xI_h(\bff m_h^j\times\hat{\Delta}_jW_h)
    }{\partial_x\bff m_h^j}$.
The sequence $\{\mathfrak Z_{j+1}\}$ is a martingale difference sequence.
By \eqref{eq:Ih-product-H1-sllg},
\eqref{eq:noise-increment-moment-sllg}, and independence of the Brownian
increment from $\mathcal F_{t_j}$, we have
\[
    \mathbb E \bigl[|\mathfrak Z_{j+1}|^p \bigr]
    \le
    C_p k^{\frac{p}{2}}
    \mathbb E\left[
      \left(
        1+\norm{\partial_x\bff m_h^j}{\bb L^2}^2
      \right)^p
    \right].
\]
The discrete BDG inequality followed by
Lemma~\ref{lem:weighted-discrete-holder} proves the result.
\end{proof}

We also have the following $p$-th moment bound.

\begin{proposition}[$p$-th moment bound]\label{prop:p-moment-stability-sllg}
Let $p\ge2$, let $\alpha,\kappa>0$, and suppose
Assumption~\ref{ass:noise-regularity-sllg} holds. Let
$\{\bff m_h^j\}_{j=0}^J$ be an adapted solution of
Algorithm~\ref{alg:midpoint-sllg}.
Then there exists a constant $C_{p,T,\alpha,\kappa,\Gamma}>0$, independent of $h$ and $k$, such that
\begin{align}
    &\mathbb E\left[
      \max_{0\le j\le J}
      \norm{\partial_x\bff m_h^j}{\bb L^2}^{2p}
    \right]
    +
    \mathbb E\left[
      \left(
        \sum_{j=0}^{J-1}
        \norm{\partial_x(\bff m_h^{j+1}-\bff m_h^j)}{\bb L^2}^2
      \right)^p
    \right]
    \nonumber\\
    &\quad+
    \mathbb E\left[
      \left(
        \sum_{j=0}^{J-1}k
        \norm{
          \bff m_h^{j+\frac12}\times\Delta_h\bff m_h^{j+1}
        }{h}^2
      \right)^p
    \right]
    \le
    C_{p,T,\alpha,\kappa,\Gamma}
    \left(
      1+
      \mathbb E\left[
        \norm{\partial_x\bff m_h^0}{\bb L^2}^{2p}
      \right]
    \right).
    \label{eq:p-moment-stability-sllg}
\end{align}
\end{proposition}

\begin{proof}
	Since the mass-lumped scheme \eqref{eq:fully-discrete-midpoint-sllg-laplacian} holds nodewise and
	$\norm{\bff m_h^{j+\frac12}}{\bb L^\infty}\le 1$, we have
	\begin{align}
		\norm{\bff m_h^{j+1}-\bff m_h^j}{h}
		&\le
		Ck
		\norm{
			\bff m_h^{j+\frac12}
			\times\Delta_h\bff m_h^{j+1}
		}{h}
		+
		C
		\norm{\hat{\Delta}_jW_h}{\bb L^\infty}.
		\label{eq:preliminary-L2-increment-sllg}
	\end{align}
	
	Testing
	\eqref{eq:fully-discrete-midpoint-sllg-laplacian} with
	$-\kappa\Delta_h\bff m_h^{j+1}$ gives
	\begin{align}
		&\frac{\kappa}{2}
		\left(
		\norm{\partial_x\bff m_h^{j+1}}{\bb L^2}^2
		-
		\norm{\partial_x\bff m_h^j}{\bb L^2}^2
		+
		\norm{
			\partial_x(\bff m_h^{j+1}-\bff m_h^j)
		}{\bb L^2}^2
		\right)
		\nonumber\\
		&\quad+
		\alpha\kappa^2k
		\norm{
			\bff m_h^{j+\frac12}
			\times\Delta_h\bff m_h^{j+1}
		}{h}^2
		=
		\kappa
		\inpro{
			\partial_xI_h
			\left(
			\bff m_h^{j+\frac12}
			\times\hat{\Delta}_jW_h
			\right)
		}{
			\partial_x\bff m_h^{j+1}
		}.
		\label{eq:p-moment-proof-energy-identity-sllg}
	\end{align}
	By using the identity
	\[
	\bff m_h^{j+\frac12}
	=
	\bff m_h^j
	+
	\frac12(\bff m_h^{j+1}-\bff m_h^j),
	\]
	after adding and subtracting a common term, the right-hand side of
	\eqref{eq:p-moment-proof-energy-identity-sllg} becomes
	\begin{align}
		&
		\kappa
		\inpro{
			\partial_xI_h(
			\bff m_h^j\times\hat{\Delta}_jW_h
			)
		}{
			\partial_x\bff m_h^j
		}
		+
		\kappa
		\inpro{
			\partial_xI_h(
			\bff m_h^j\times\hat{\Delta}_jW_h
			)
		}{
			\partial_x(\bff m_h^{j+1}-\bff m_h^j)
		}
		\nonumber\\
		&\quad+
		\frac{\kappa}{2}
		\inpro{
			\partial_xI_h
			\left(
			(\bff m_h^{j+1}-\bff m_h^j)
			\times\hat{\Delta}_jW_h
			\right)
		}{
			\partial_x\bff m_h^j
		}
		\nonumber\\
		&\quad+
		\frac{\kappa}{2}
		\inpro{
			\partial_xI_h
			\left(
			(\bff m_h^{j+1}-\bff m_h^j)
			\times\hat{\Delta}_jW_h
			\right)
		}{
			\partial_x(\bff m_h^{j+1}-\bff m_h^j)
		} =: \mathrm{I}_1+\ldots+\mathrm{I}_4.
		\label{eq:p-moment-proof-noise-splitting-sllg}
	\end{align}
	We now estimate each term above, except for $\mathrm{I}_1$. The second term satisfies
	\begin{align*}
		\left|
		\mathrm{I}_2
		\right|
		\le
		\frac{\kappa}{12}
		\norm{
			\partial_x(\bff m_h^{j+1}-\bff m_h^j)
		}{\bb L^2}^2
		+
		C
		\norm{\hat{\Delta}_jW_h}{\bb W^{1,\infty}}^2
		\Bigl(
		1+
		\norm{\partial_x\bff m_h^j}{\bb L^2}^2
		\Bigr).
	\end{align*}
	For the term $\mathrm{I}_3$,
	\eqref{eq:Ih-product-H1-L2refined-sllg} and
	\eqref{eq:preliminary-L2-increment-sllg} give
	\begin{align*}
		\left|
		\mathrm{I}_3
		\right|
		&\le
		\frac{\kappa}{12}
		\norm{
			\partial_x(\bff m_h^{j+1}-\bff m_h^j)
		}{\bb L^2}^2
		+
		\frac{\alpha\kappa^2}{4}k
		\norm{
			\bff m_h^{j+\frac12}
			\times\Delta_h\bff m_h^{j+1}
		}{h}^2
		+
		C
		\norm{\hat{\Delta}_jW_h}{\bb W^{1,\infty}}^2
		\Bigl(
		1+\norm{\partial_x\bff m_h^j}{\bb L^2}^2
		\Bigr).
	\end{align*}
	Finally, using
	\eqref{eq:discrete-cross-cancellation-sllg} and an elementary inequality
	$\norm{\bff m_h^{j+1}-\bff m_h^j}{\bb L^\infty} \le 2$,
	we obtain
	\begin{align*}
		\left|\mathrm{I}_4 \right|
		&\le
		\frac{\kappa}{12}
		\norm{
			\partial_x(\bff m_h^{j+1}-\bff m_h^j)
		}{\bb L^2}^2
		+
		C\norm{\hat{\Delta}_jW_h}{\bb W^{1,\infty}}^2.
	\end{align*}
	
	Combining these estimates and absorbing the corresponding terms into
	the right-hand side of \eqref{eq:p-moment-proof-energy-identity-sllg}, noting \eqref{eq:p-moment-proof-noise-splitting-sllg}, yields
	\begin{align*}
		&\frac{\kappa}{2}
		\left(
		\norm{\partial_x\bff m_h^{j+1}}{\bb L^2}^2
		-
		\norm{\partial_x\bff m_h^j}{\bb L^2}^2
		\right)
		+
		\frac{\kappa}{4}
		\norm{
			\partial_x(\bff m_h^{j+1}-\bff m_h^j)
		}{\bb L^2}^2
		+
		\frac{\alpha\kappa^2}{2}k
		\norm{
			\bff m_h^{j+\frac12}
			\times\Delta_h\bff m_h^{j+1}
		}{h}^2
		\nonumber\\
		&\qquad
		\le
		\kappa
		\inpro{
			\partial_xI_h(
			\bff m_h^j\times\hat{\Delta}_jW_h
			)
		}{
			\partial_x\bff m_h^j
		}
		+
		C
		\norm{\hat{\Delta}_jW_h}{\bb W^{1,\infty}}^2
		\left(
		1+
		\norm{\partial_x\bff m_h^j}{\bb L^2}^2
		\right).
	\end{align*}
	
	Summing up to an arbitrary time level, taking the maximum over the
	partial sums, and then taking the $p$-th power gives
	\begin{align*}
		&\max_{0\le r\le J}
		\norm{\partial_x\bff m_h^r}{\bb L^2}^{2p}
		+
		\left(
		\sum_{j=0}^{J-1}
		\norm{
			\partial_x(\bff m_h^{j+1}-\bff m_h^j)
		}{\bb L^2}^2
		\right)^p
		+
		\left(
		k\sum_{j=0}^{J-1}
		\norm{
			\bff m_h^{j+\frac12}
			\times\Delta_h\bff m_h^{j+1}
		}{h}^2
		\right)^p
		\nonumber\\
		&\qquad
		\le
		C_p
		\norm{\partial_x\bff m_h^0}{\bb L^2}^{2p}
		+
		C_p
		\max_{0\le r\le J}
		\left|
		\sum_{j=0}^{r-1}
		\kappa
		\inpro{
			\partial_xI_h(
			\bff m_h^j\times\hat{\Delta}_jW_h
			)
		}{
			\partial_x\bff m_h^j
		}
		\right|^p
		\nonumber\\
		&\qquad\quad+
		C_p
		\left(
		\sum_{j=0}^{J-1}
		\norm{\hat{\Delta}_jW_h}{\bb W^{1,\infty}}^2
		\left(
		1+
		\norm{\partial_x\bff m_h^j}{\bb L^2}^2
		\right)
		\right)^p.
	\end{align*}
	
	After taking expectations,
	Lemma~\ref{lem:BDG-stochastic-energy-input} controls the martingale
	term. Moreover, Lemma~\ref{lem:weighted-discrete-holder},
	the independence of $\hat{\Delta}_jW_h$ from
	$\mathcal F_{t_j}$, and \eqref{eq:noise-increment-moment-sllg} imply
	\begin{align*}
		&\mathbb E
		\left[
		\left(
		\sum_{j=0}^{J-1}
		\norm{\hat{\Delta}_jW_h}{\bb W^{1,\infty}}^2
		\left(
		1+
		\norm{\partial_x\bff m_h^j}{\bb L^2}^2
		\right)
		\right)^p
		\right]
		\le
		C_{p,T,\Gamma}
		\left(
		1+
		\sum_{j=0}^{J-1}
		k\,
		\mathbb E
		\left[
		\norm{\partial_x\bff m_h^j}{\bb L^2}^{2p}
		\right]
		\right).
	\end{align*}
	An application of the discrete Gronwall lemma then proves~\eqref{eq:p-moment-stability-sllg}.
\end{proof}

The following discrete increment estimates will be essential for our analysis.

\begin{lemma} \label{lem:mh-increment-estimates-sllg}
	Assume \eqref{eq:noise-W1inf-l1-assumption-sllg} holds. Then there exists a constant $C>0$, independent of $h$ and $k$, such that
	\begin{equation}\label{eq:mh-increment-second-moment-sllg}
		\mathbb E
		\left[
		\sum_{j=0}^{J-1}
		\norm{\bff m_h^{j+1}-\bff m_h^j}{h}^2
		\right]
		+
		\frac1k \mathbb E
		\left[
		\sum_{j=0}^{J-1}
		\norm{\bff m_h^{j+1}-\bff m_h^j}{h}^4
		\right]
		\le C.
	\end{equation}
	Consequently,
	\begin{equation}\label{eq:mh-increment-fourth-root-sllg}
		\sum_{j=0}^{J-1}
		\left(
		\mathbb E
		\left[
		\norm{\bff m_h^{j+1}-\bff m_h^j}{h}^4
		\right]
		\right)^{\frac12}
		\le C .
	\end{equation}
\end{lemma}

\begin{proof}
	Our starting point is the estimate~\eqref{eq:preliminary-L2-increment-sllg}.
	Taking expectations, using \eqref{eq:p-moment-stability-sllg} with
	$p=2$ and Jensen's inequality for the first term, and using
	\eqref{eq:noise-increment-moment-sllg} with $q=1$ for the second term,
	we obtain
	\[
	\mathbb E
	\left[
	\sum_{j=0}^{J-1}
	\norm{\bff m_h^{j+1}-\bff m_h^j}{h}^2
	\right]
	\le
	Ck
	+
	C\sum_{j=0}^{J-1}k
	\le C,
	\]
	thus proving the estimate for the first term in \eqref{eq:mh-increment-second-moment-sllg}.
	
	Next, raising~\eqref{eq:preliminary-L2-increment-sllg} to the fourth
	power gives
	\begin{align}
		\norm{\bff m_h^{j+1}-\bff m_h^j}{h}^4
		&\le
		Ck^4
		\norm{
			\bff m_h^{j+\frac12}\times\Delta_h\bff m_h^{j+1}
		}{h}^4
		+
		C
		\norm{\hat{\Delta}_jW_h}{\bb L^\infty}^4 .
		\label{eq:mh-increment-fourth-pointwise-sllg}
	\end{align}
	Summing from $j=0$ to $J-1$, taking expectations, and using an elementary inequality
	\[
	k^2
	\sum_{j=0}^{J-1}
	\norm{
		\bff m_h^{j+\frac12}\times\Delta_h\bff m_h^{j+1}
	}{h}^4
	\le
	\left(
	k
	\sum_{j=0}^{J-1}
	\norm{
		\bff m_h^{j+\frac12}\times\Delta_h\bff m_h^{j+1}
	}{h}^2
	\right)^2,
	\]
	we obtain
	\begin{align}
		\mathbb E
		\left[
		\sum_{j=0}^{J-1}
		\norm{\bff m_h^{j+1}-\bff m_h^j}{h}^4
		\right]
		&\le
		Ck^2
		\mathbb E
		\left[
		\left(
		k
		\sum_{j=0}^{J-1}
		\norm{
			\bff m_h^{j+\frac12}\times\Delta_h\bff m_h^{j+1}
		}{h}^2
		\right)^2
		\right]
		+
		C
		\sum_{j=0}^{J-1}
		\mathbb E
		\left[
		\norm{\hat{\Delta}_jW_h}{\bb L^\infty}^4
		\right]
		\nonumber\\
		&\le
		Ck^2
		+
		C\sum_{j=0}^{J-1}k^2
		\le
		Ck,
		\label{eq:mh-increment-fourth-proof-sllg}
	\end{align}
	where we also used the moment bound \eqref{eq:p-moment-stability-sllg} with $p=2$ and~\eqref{eq:noise-increment-moment-sllg} with $q=2$.
	This proves \eqref{eq:mh-increment-second-moment-sllg}.
	
	Finally, by Cauchy's inequality in the time index,
	\begin{align*}
		\sum_{j=0}^{J-1}
		\left(
		\mathbb E
		\left[
		\norm{\bff m_h^{j+1}-\bff m_h^j}{h}^4
		\right]
		\right)^{\frac12}
		&\le
		J^{\frac12}
		\left(
		\sum_{j=0}^{J-1}
		\mathbb E
		\left[
		\norm{\bff m_h^{j+1}-\bff m_h^j}{h}^4
		\right]
		\right)^{\frac12}
		\le
		\left(\frac{T}{k}\right)^{\frac12}
		(Ck)^{\frac12}
		\le C .
	\end{align*}
	This proves \eqref{eq:mh-increment-fourth-root-sllg} and completes the proof of the lemma.
\end{proof}

\subsection{Localised error estimate}
\label{subsec:error-estimate-sllg}

In this subsection, we prove a localised \emph{a priori} error estimate in the natural energy norm.
For brevity, we write
\[
\bff m^j:=\bff m(t_j),
\qquad
\overline{\bff m}^{j+\frac12}
:=
\frac12(\bff m^{j+1}+\bff m^j),
\qquad
j=0,\ldots,J-1.
\]

For $j=0,\ldots,J$, define
\begin{equation}\label{eq:error-splitting-sllg}
	\bff\rho_h^j
	:=
	\bff m^j-I_h\bff m^j,
	\qquad
	\bff\theta_h^j
	:=
	I_h\bff m^j-\bff m_h^j,
\end{equation}
so that $\bff m^j-\bff m_h^j=\bff\rho_h^j+\bff\theta_h^j$.
By \eqref{eq:Ih-H1-approx-sllg}, the interpolation error $\bff{\rho}_h^j$ satisfies
\begin{align}
	&\max_{0\le j\le J}
	\norm{\bff\rho_h^j}{\bb L^2}^2
	+
	k\sum_{j=0}^{J-1}
	\norm{\partial_x\bff\rho_h^{j+1}}{\bb L^2}^2
	\le
	Ch^2 \norm{\bff m}{L^\infty(0,T;\bb H^2)}^2 .
	\label{eq:rho-interpolation-error-sllg}
\end{align}

To derive the error equation, we first compare the deterministic drift of
the exact solution with the corresponding discrete drift evaluated at the
projected exact solution. For $j=0,\ldots,J-1$, define the drift residual
$\mathcal R_h^{j+1}\in\bb V_h'$ by
\begin{equation}\label{eq:det-residual}
\mathcal R_h^{j+1}
:=
\mathcal R_{h,\mathrm d}^{j+1}
+
\mathcal R_{h,\mathrm p}^{j+1},
\end{equation}
where the damping residual $\mathcal R_{h,\mathrm d}^{j+1}$ measures the
difference between the discrete midpoint approximation of the damping
drift and its exact time integral:
\begin{align*}
	\mathcal R_{h,\mathrm d}^{j+1}(\bff\phi_h)
	&:=
	\alpha\kappa k
	\inpro{
		I_h\overline{\bff m}^{j+\frac12}
		\times
		\left[
		I_h\overline{\bff m}^{j+\frac12}
		\times
		\Delta_h I_h\bff m^{j+1}
		\right]
	}{
		\bff\phi_h
	}_h
	\nonumber
	\\
	&\quad
	-
	\alpha\kappa
	\int_{t_j}^{t_{j+1}}
	\inpro{
		I_h
		\left[
		\bff m(s)\times
		\left(
		\bff m(s)\times\partial_{xx}\bff m(s)
		\right)
		\right]
	}{
		\bff\phi_h
	}_h
	\ds,
\end{align*}
and the precession residual $\mathcal R_{h,\mathrm p}^{j+1}$ is defined
analogously by
\begin{align*}
	\mathcal R_{h,\mathrm p}^{j+1}(\bff\phi_h)
	&:=
	-\kappa k
	\inpro{
		I_h\overline{\bff m}^{j+\frac12}
		\times
		\Delta_h I_h\bff m^{j+1}
	}{
		\bff\phi_h
	}_h
	+
	\kappa
	\int_{t_j}^{t_{j+1}}
	\inpro{
		I_h
		\left[
		\bff m(s)\times\partial_{xx}\bff m(s)
		\right]
	}{
		\bff\phi_h
	}_h
	\ds .
\end{align*}

Note that since $I_h\bff m(s)$ and $\bff m(s)$, as well as
$\bff g_{\ell,h}$ and $\bff g_\ell$, agree at every finite element node,
the mass-lumped inner product gives
$\inpro{
	I_h\bff m(s)\times\bff g_{\ell,h}
}{
	\bff\phi_h
}_h
=
\inpro{
	\bff m(s)\times\bff g_\ell
}{
	\bff\phi_h
}_h$.
Hence, after nodal projection, the stochastic term may equivalently be
written using $W_h$.
Integrating the exact equation over $[t_j,t_{j+1}]$, applying the nodal
interpolant $I_h$, then adding and subtracting the two discrete drift
terms appearing above shows that the projected exact solution
$I_h\bff m^j$ satisfies
\begin{align}
	&\inpro{
		I_h\bff m^{j+1}-I_h\bff m^j
	}{
		\bff\phi_h
	}_h
	+
	\alpha\kappa k
	\inpro{
		I_h\overline{\bff m}^{j+\frac12}
		\times
		\left[
		I_h\overline{\bff m}^{j+\frac12}
		\times
		\Delta_h I_h\bff m^{j+1}
		\right]
	}{
		\bff\phi_h
	}_h
	\nonumber
	\\
	&\quad
	-
	\kappa k
	\inpro{
		I_h\overline{\bff m}^{j+\frac12}
		\times
		\Delta_h I_h\bff m^{j+1}
	}{
		\bff\phi_h
	}_h
	-
	\inpro{
		\int_{t_j}^{t_{j+1}}
		I_h\bff m(s)\times\circ\,\dW_h(s)
	}{
		\bff\phi_h
	}_h
	=
	\mathcal R_h^{j+1}(\bff\phi_h).
	\label{eq:projected-exact-discrete-equation-sllg}
\end{align}
Therefore, \eqref{eq:projected-exact-discrete-equation-sllg} has the same
deterministic structure as the scheme~\eqref{eq:midpoint-def-sllg}, up to the drift residual~$\mathcal R_h^{j+1}$. Subtracting the numerical scheme thus yields the discrete error equation directly.
The stochastic term is kept in its exact integral form, since its midpoint
consistency error requires a more delicate argument. Here, $\mathcal R_h^{j+1}$ contains only the drift
consistency error, while the stochastic part is treated separately in
Lemma~\ref{lem:stochastic-error-term-sllg}.

The following lemma estimates this drift residual in the dual norm
\eqref{eq:dual-Hminus1h-def-sllg}. This is the estimate needed when
$\mathcal R_h^{j+1}$ is paired with the error $\bff\theta_h^{j+1}$ in the
localised energy argument later.

\begin{lemma}[Drift consistency]\label{lem:consistency-sllg}
	Let $\gamma\in(0,\frac12)$, and let $\mathcal R_h^{j+1}$ be the drift
	residual defined by \eqref{eq:det-residual}. Then
	\begin{equation}\label{eq:consistency-bound-sllg}
		\mathbb E\left[
		\sum_{j=0}^{J-1}
		\frac1k\norm{\mathcal R_h^{j+1}}{-1,h}^2
		\right]
		\le
		C_\gamma\left(h^2+k^{2\gamma}\right).
	\end{equation}
\end{lemma}

\begin{proof}
	We split
	\[
	\mathcal R_{h,\mathrm p}^{j+1}
	=
	\widetilde{\mathcal R}_{h,\mathrm p}^{j+1}
	+\mathcal Q_{h,\mathrm p}^{j+1},
	\qquad
	\mathcal R_{h,\mathrm d}^{j+1}
	=
	\widetilde{\mathcal R}_{h,\mathrm d}^{j+1}
	+\mathcal Q_{h,\mathrm d}^{j+1},
	\]
	where the quadrature residuals
	\begin{align*}
		\mathcal Q_{h,\mathrm p}^{j+1}(\bff\phi_h)
		&:={}
		\kappa\int_{t_j}^{t_{j+1}}
		\left[
		\inpro{
			I_h\left(
			\bff m(s)\times\partial_{xx}\bff m(s)
			\right)
		}{\bff\phi_h}_h
		-\inpro{
			\bff m(s)\times\partial_{xx}\bff m(s)
		}{\bff\phi_h}
		\right]\ds,
		\\
		\mathcal Q_{h,\mathrm d}^{j+1}(\bff\phi_h)
		&:={}
		-\alpha\kappa\int_{t_j}^{t_{j+1}}
		\left[
		\inpro{
			I_h\left[
			\bff m(s)\times
			\left(
			\bff m(s)\times\partial_{xx}\bff m(s)
			\right)
			\right]
		}{\bff\phi_h}_h
		-\inpro{
			\bff m(s)\times
			\left(
			\bff m(s)\times\partial_{xx}\bff m(s)
			\right)
		}{\bff\phi_h}
		\right]\ds.
	\end{align*}
	Thus, the tilded residuals $\widetilde{\mathcal R}_{h,\mathrm p}^{j+1}$ and $\widetilde{\mathcal R}_{h,\mathrm d}^{j+1}$ are obtained by replacing the exact mass-lumped pairings by standard $\bb L^2$-pairings.
	
	
	We first estimate the precession residual. By the definition of $\Delta_h$ and applying integration by parts, we add and subtract a common term to obtain
	$\widetilde{\mathcal R}_{h,\mathrm p}^{j+1}
	=\mathcal R_{h,\mathrm p,1}^{j+1}
	+\mathcal R_{h,\mathrm p,2}^{j+1}$, where
	\begin{align*}
		\mathcal R_{h,\mathrm p,1}^{j+1}(\bff\phi_h)
		&:={}
		\kappa k
		\left[
		\inpro{\partial_xI_h\bff m^{j+1}}{
			\partial_xI_h\left(
			\bff\phi_h\times I_h\overline{\bff m}^{j+\frac12}
			\right)
		}
		-
		\inpro{\partial_x\bff m^{j+1}}{
			\partial_x\left(
			\bff\phi_h\times\overline{\bff m}^{j+\frac12}
			\right)
		}
		\right],
		\\
		\mathcal R_{h,\mathrm p,2}^{j+1}(\bff\phi_h)
		&:={}
		\kappa\int_{t_j}^{t_{j+1}}
		\left[
		\inpro{\partial_x\bff m^{j+1}}{
			\partial_x\left(
			\bff\phi_h\times\overline{\bff m}^{j+\frac12}
			\right)
		}
		-
		\inpro{\partial_x\bff m(s)}{
			\partial_x\left(\bff\phi_h\times\bff m(s)\right)
		}
		\right]\ds.
	\end{align*}
	Using the interpolation estimate, \eqref{eq:Ih-product-H1-sllg},
	and \eqref{eq:product-interpolation-single-sllg}, we obtain
	\begin{equation}\label{eq:precession-space-dual-bound-weak-sllg}
		\norm{\mathcal R_{h,\mathrm p,1}^{j+1}}{-1,h}
		\le
		Ckh
		\norm{\bff m^{j+1}}{\bb H^2}
		\norm{\overline{\bff m}^{j+\frac12}}{\bb H^2}.
	\end{equation}
	Furthermore, the one-dimensional $\bb{H}^1$-product estimate gives, for
	$s\in[t_j,t_{j+1}]$,
	\begin{align*}
		&\left|
		\inpro{\partial_x\bff m^{j+1}}{
			\partial_x(\bff\phi_h\times\overline{\bff m}^{j+\frac12})
		}
		-
		\inpro{\partial_x\bff m(s)}{
			\partial_x(\bff\phi_h\times\bff m(s))
		}
		\right|
		\\
		&\qquad
		\le
		C\left(1+\norm{\bff m}{L^\infty(0,T;\bb H^2)}^2\right)
		\left(
		\norm{\bff m(s)-\bff m^{j+1}}{\bb H^1}
		+
		\norm{\bff m(s)-\overline{\bff m}^{j+\frac12}}{\bb H^1}
		\right)
		\norm{\bff\phi_h}{1,h}.
	\end{align*}
	Consequently,
	\begin{equation}\label{eq:precession-time-dual-holder-bound-weak-sllg}
		\norm{\mathcal R_{h,\mathrm p,2}^{j+1}}{-1,h}
		\le
		Ck^{1+\gamma}
		\left(1+\norm{\bff m}{L^\infty(0,T;\bb H^2)}^2\right)
		\norm{\bff m}{C^\gamma([0,T];\bb H^1)}.
	\end{equation}
	Hence, from \eqref{eq:precession-space-dual-bound-weak-sllg} and \eqref{eq:precession-time-dual-holder-bound-weak-sllg}, we get
	\begin{equation}\label{eq:precession-residual-bound-sllg}
		\mathbb E\left[
		\sum_{j=0}^{J-1}\frac1k
		\norm{\widetilde{\mathcal R}_{h,\mathrm p}^{j+1}}{-1,h}^2
		\right]
		\le
		C_\gamma\left(h^2+k^{2\gamma}\right).
	\end{equation}
	
	To estimate $\widetilde{\mathcal R}_{h,\mathrm d}^{j+1}$, we observe that by the scalar triple-product identity and the definition of $\Delta_h$,
	\begin{align*}
		&
		\inpro{
			I_h\overline{\bff m}^{j+\frac12}
			\times\left[
			I_h\overline{\bff m}^{j+\frac12}
			\times\Delta_hI_h\bff m^{j+1}
			\right]
		}{\bff\phi_h}_h
		=
		-
		\inpro{\partial_xI_h\bff m^{j+1}}{
			\partial_xI_h\left[
			\left(
			\bff\phi_h\times I_h\overline{\bff m}^{j+\frac12}
			\right)
			\times I_h\overline{\bff m}^{j+\frac12}
			\right]
		}.
	\end{align*}
	Proceeding similarly as before, but using
	\eqref{eq:product-interpolation-double-sllg}, yields a decomposition
	$\widetilde{\mathcal R}_{h,\mathrm d}^{j+1}
	=\mathcal R_{h,\mathrm d,1}^{j+1}
	+\mathcal R_{h,\mathrm d,2}^{j+1}$ with
	\begin{align*}
		\norm{\mathcal R_{h,\mathrm d,1}^{j+1}}{-1,h}
		&\le
		Ckh
		\norm{\bff m^{j+1}}{\bb H^2}
		\norm{\overline{\bff m}^{j+\frac12}}{\bb H^2}^2,
		\\
		\norm{\mathcal R_{h,\mathrm d,2}^{j+1}}{-1,h}
		&\le
		Ck^{1+\gamma}
		\left(1+\norm{\bff m}{L^\infty(0,T;\bb H^2)}^3\right)
		\norm{\bff m}{C^\gamma([0,T];\bb H^1)}.
	\end{align*}
	Therefore,
	\begin{equation}\label{eq:damping-time-bound}
		\mathbb E\left[
		\sum_{j=0}^{J-1}\frac1k
		\norm{\widetilde{\mathcal R}_{h,\mathrm d}^{j+1}}{-1,h}^2
		\right]
		\le
		C_\gamma\left(h^2+k^{2\gamma}\right).
	\end{equation}
	
	It remains to estimate the quadrature residuals $\mathcal Q_{h,\mathrm p}$ and $\mathcal Q_{h,\mathrm d}$. By \eqref{eq:smooth-lumped-quadrature-sllg}, we have
	\begin{align*}
		\norm{\mathcal Q_{h,\mathrm p}^{j+1}}{-1,h}
		+\norm{\mathcal Q_{h,\mathrm d}^{j+1}}{-1,h}
		&\leq
		Ch\int_{t_j}^{t_{j+1}}
		\left(
		\norm{
			\bff m(s)\times\partial_{xx}\bff m(s)
		}{\bb H^1}
		+\norm{
			\bff m(s)\times
			\left(
			\bff m(s)\times\partial_{xx}\bff m(s)
			\right)
		}{\bb H^1}
		\right)\ds
		\\
		&\leq
		Ckh  \left(1+\norm{\bff m}{L^\infty(0,T;\bb H^3)}^3\right),
	\end{align*}
	where in the last step we used the one-dimensional $\bb{H}^1$-product estimate and the regularity of $\bff{m}$. This implies
	\begin{equation}\label{eq:quadrature-residual-bound-sllg}
		\mathbb E\left[
		\sum_{j=0}^{J-1}\frac1k
		\left(
		\norm{\mathcal Q_{h,\mathrm p}^{j+1}}{-1,h}^2
		+\norm{\mathcal Q_{h,\mathrm d}^{j+1}}{-1,h}^2
		\right)
		\right]
		\le
		Ch^2.
	\end{equation}
	Combining \eqref{eq:precession-residual-bound-sllg},
	\eqref{eq:damping-time-bound}, and
	\eqref{eq:quadrature-residual-bound-sllg} proves
	\eqref{eq:consistency-bound-sllg}, as required.
\end{proof}

For each $\Lambda>0$ and $j=0,\ldots,J$, define a subset
\begin{equation}\label{eq:localisation-set-sllg}
	\Omega_{\Lambda,j}
	:=
	\left\{
	\max_{0\le r\le j}
	\norm{\partial_x\bff m_h^r}{\bb L^2}^4
	+
	\sup_{0\le t\le t_j}
	\norm{\bff m(t)}{\bb H^2}^4
	\le\Lambda
	\right\}.
\end{equation}
Then, $\Omega_{\Lambda,j+1}\subset\Omega_{\Lambda,j}$ and
$\mathbf1_{\Omega_{\Lambda,j}}$ is $\mathcal F_{t_j}$-measurable. Moreover, by \eqref{eq:exact-solution-regularity-sllg} and \eqref{eq:p-moment-stability-sllg}, there exists a constant $C>0$ independent of $h$, $k$, and $\Lambda$ such that
\begin{equation}\label{eq:localisation-probability-sllg}
	\mathbb P\left[\Omega_{\Lambda,J}^{\complement}\right]
	\le
	\frac{C}{\Lambda}.
\end{equation}

The consistency error in the stochastic part is handled by the following technical lemma.

\begin{lemma}[Stochastic error term]\label{lem:stochastic-error-term-sllg}
	Let $\varepsilon\in(0,1)$ and $\Lambda>0$ be fixed. Let $\bff{m}$ be the strong solution of \eqref{eq:sllg-model} and $\{\bff m_h^j\}_{j=0}^J$ be an adapted solution of
	\eqref{eq:fully-discrete-midpoint-sllg-laplacian}. Let $W_h$ be given by \eqref{eq:gh Wh}. For $r=0,1,\ldots,J$, define
	\begin{align}
		\mathcal S_r
		&:=
		\sum_{j=0}^{r-1}
		\mathbf1_{\Omega_{\Lambda,j}}
		\left[
		\inpro{
			\displaystyle
			\int_{t_j}^{t_{j+1}}
			I_h\bff m(s)\times\circ\,\dW_h(s)
		}{
			\bff\theta_h^{j+1}
		}_h
		-
		\inpro{
			\bff m_h^{j+\frac12}\times\hat{\Delta}_jW_h
		}{
			\bff\theta_h^{j+1}
		}_h
		\right].
		\label{eq:stochastic-error-term-definition-sllg}
	\end{align}
	Then, for every $\delta>0$ and $\eta>0$, there exists a constant
	$C_{\delta,\eta,\varepsilon,\Lambda,T}>0$, independent of $h$ and $k$,
	such that
	\begin{align}
		\mathbb E
		\left[
		\sup_{0\le r\le J}
		|\mathcal S_r|
		\right]
		&\le
		\delta
		\mathbb E
		\left[
		\sum_{j=0}^{J-1}
		\mathbf1_{\Omega_{\Lambda,j}}
		\norm{\bff\theta_h^{j+1}-\bff\theta_h^j}{h}^2
		\right]
		+
		\eta
		\mathbb E
		\left[
		\sup_{0\le j\le J}
		\mathbf1_{\Omega_{\Lambda,j}}
		\norm{\bff\theta_h^j}{h}^2
		\right]
		\nonumber
		\\
		&\quad
		+
		C_{\delta,\eta,\varepsilon,\Lambda,T}
		\left(
		k^{1-\varepsilon}
		+
		k
		\sum_{j=0}^{J-1}
		\mathbb E
		\left[
		\mathbf1_{\Omega_{\Lambda,j}}
		\norm{\bff\theta_h^{j+1}}{h}^2
		\right]
		\right).
		\label{eq:stochastic-error-term-estimate-sllg}
	\end{align}
\end{lemma}

\begin{proof}
	We begin the proof with several observations. First, note that applying \eqref{eq:exact-solution-regularity-sllg} with the H\"older exponent $\frac12 (1-\varepsilon) \in(0,\frac12)$ gives
	\begin{align}
		&\sum_{j=0}^{J-1}\int_{t_j}^{t_{j+1}}
		\mathbb E\left[
		\norm{\bff m(s)-\bff m^j}{\bb H^1}^2
		+\norm{\bff m(s)-\bff m^{j+1}}{\bb H^1}^2
		\right]\ds
		\le
		C_\varepsilon k^{1-\varepsilon}.
		\label{eq:exact-time-increment-sum-sllg}
	\end{align}
	Next, the Stratonovich integral is rewritten as a vector-valued It\^o
	integral to obtain
	\begin{align}
		\inpro{
			\int_{t_j}^{t_{j+1}} I_h\bff m(s)\times\circ\, \dW_h(s)
		}{
			\bff\theta_h^{j+1}
		}_h
		\nonumber
		&
		=
		\inpro{
			\sum_{\ell=1}^{\infty}
			\int_{t_j}^{t_{j+1}} I_h\bff m(s)\times\bff g_{\ell,h}\,\dbeta_\ell(s)
		}{
			\bff\theta_h^{j+1}
		}_h
		\nonumber
		\\
		&\quad
		+
		\frac12
		\sum_{\ell=1}^{\infty}
		\int_{t_j}^{t_{j+1}}
		\inpro{
			\left(I_h\bff m(s)\times\bff g_{\ell,h}\right)\times\bff g_{\ell,h}
		}{
			\bff\theta_h^{j+1}
		}_h\,\ds .
		\label{eq:stratonovich-ito-paired-sllg}
	\end{align}
	On the other hand,
	\begin{align}
		\inpro{
			\bff m_h^{j+\frac12}\times\hat{\Delta}_jW_h
		}{
			\bff\theta_h^{j+1}
		}_h
		&=
		\sum_{\ell=1}^{\infty}
		\hat{\Delta}_j\beta_\ell
		\inpro{
			\bff m_h^j\times\bff g_{\ell,h}
		}{
			\bff\theta_h^{j+1}
		}_h
		\nonumber
		\\
		&\quad
		+
		\frac12
		\sum_{\ell=1}^{\infty}
		\hat{\Delta}_j\beta_\ell
		\inpro{
			(\bff m_h^{j+1}-\bff m_h^j)\times\bff g_{\ell,h}
		}{
			\bff\theta_h^{j+1}
		}_h .
		\label{eq:midpoint-noise-expanded-sllg}
	\end{align}
	
	We aim rewrite the last term in \eqref{eq:midpoint-noise-expanded-sllg}.
	The fully discrete scheme \eqref{eq:fully-discrete-midpoint-sllg-laplacian} implies the following identity, which holds at every node:
	\begin{align*}
		\bff m_h^{j+1}-\bff m_h^j
		&=
		-\alpha\kappa k\,\bff m_h^{j+\frac12}\times
		\left(\bff m_h^{j+\frac12}\times\Delta_h\bff m_h^{j+1}\right)
		+
		\kappa k\,\bff m_h^{j+\frac12}\times\Delta_h\bff m_h^{j+1}
		+
		\bff m_h^{j+\frac12}\times\hat{\Delta}_jW_h.
	\end{align*}
	Therefore, using this in the last mass-lumped inner product term in \eqref{eq:midpoint-noise-expanded-sllg} gives
	\begin{align}
		&\frac12
		\sum_{\ell=1}^{\infty}
		\hat{\Delta}_j\beta_\ell
		\inpro{
			(\bff m_h^{j+1}-\bff m_h^j)\times\bff g_{\ell,h}
		}{
			\bff\theta_h^{j+1}
		}_h
		\nonumber
		\\
		&\quad
		=
		\frac{k}{2}
		\sum_{\ell=1}^{\infty}
		\inpro{
			\left(\bff m_h^{j+\frac12}\times\bff g_{\ell,h}\right)\times\bff g_{\ell,h}
		}{
			\bff\theta_h^{j+1}
		}_h
		+
		\frac12
		\sum_{\ell=1}^{\infty}
		\left(|\hat{\Delta}_j\beta_\ell|^2-k\right)
		\inpro{
			\left(\bff m_h^{j+\frac12}\times\bff g_{\ell,h}\right)\times\bff g_{\ell,h}
		}{
			\bff\theta_h^{j+1}
		}_h
		\nonumber
		\\
		&\qquad
		+
		\frac{\kappa k}{2}
		\sum_{\ell=1}^{\infty}
		\hat{\Delta}_j\beta_\ell
		\inpro{
			\left(\bff m_h^{j+\frac12}\times\Delta_h\bff m_h^{j+1}
			-
			\alpha\bff m_h^{j+\frac12}\times
			\left(\bff m_h^{j+\frac12}\times\Delta_h\bff m_h^{j+1}\right)\right)\times\bff g_{\ell,h}
		}{
			\bff\theta_h^{j+1}
		}_h
		\nonumber
		\\
		&\qquad
		+
		\frac12
		\sum_{\ell=1}^{\infty}
		\sum_{\substack{q=1\\q\ne\ell}}^{\infty}
		\hat{\Delta}_j\beta_\ell\,\hat{\Delta}_j\beta_q
		\inpro{
			\left(\bff m_h^{j+\frac12}\times\bff g_{q,h}\right)\times\bff g_{\ell,h}
		}{
			\bff\theta_h^{j+1}
		}_h .
		\label{eq:midpoint-noise-second-order-expanded-sllg}
	\end{align}
	Combining \eqref{eq:stratonovich-ito-paired-sllg}--
	\eqref{eq:midpoint-noise-second-order-expanded-sllg}, and using the fact that
	$I_h\bff m^j-\bff m_h^j=\bff\theta_h^j$, we can decompose \eqref{eq:stochastic-error-term-definition-sllg} as
	$\mathcal S_r
	=
	\sum_{\nu=1}^6
	\mathcal S_r^{(\nu)}$,
	where
	\begin{align}
		\mathcal S_r^{(1)}
		&:=
		\sum_{j=0}^{r-1}
		\mathbf1_{\Omega_{\Lambda,j}}
		\inpro{
			\displaystyle
			\sum_{\ell=1}^{\infty}
			\int_{t_j}^{t_{j+1}}
			I_h(\bff m(s)-\bff m^j)
			\times\bff g_{\ell,h}\,\dbeta_\ell(s)
		}{
			\bff\theta_h^{j+1}
		}_h .
		\label{eq:S1-def-sllg}
		\\
		\mathcal S_r^{(2)}
		&:=
		\sum_{j=0}^{r-1}
		\mathbf1_{\Omega_{\Lambda,j}}
		\inpro{\bff\theta_h^j\times\hat{\Delta}_jW_h}{\bff\theta_h^{j+1}}_h,
		\label{eq:S2-def-sllg}
		\\
		\mathcal S_r^{(3)}
		&:=
		\frac12
		\sum_{j=0}^{r-1}
		\mathbf1_{\Omega_{\Lambda,j}}
		\sum_{\ell=1}^{\infty}
		\int_{t_j}^{t_{j+1}}
		\inpro{\left(\left(I_h\bff m(s)-\bff m_h^{j+\frac12}\right) \times\bff g_{\ell,h}\right)\times\bff g_{\ell,h}}{\bff\theta_h^{j+1}}_h
		\,\ds,
		\label{eq:S3-def-sllg}
		\\
		\mathcal S_r^{(4)}
		&:=
		-\frac12
		\sum_{j=0}^{r-1}
		\mathbf1_{\Omega_{\Lambda,j}}
		\sum_{\ell=1}^{\infty}
		\left(|\hat{\Delta}_j\beta_\ell|^2-k\right)
		\inpro{\left(\bff m_h^{j+\frac12}\times\bff g_{\ell,h}\right)\times\bff g_{\ell,h}}{\bff\theta_h^{j+1}}_h,
		\label{eq:S4-def-sllg}
		\\
		\mathcal S_r^{(5)}
		&:=
		-\frac{\kappa k}{2}
		\sum_{j=0}^{r-1}
		\mathbf1_{\Omega_{\Lambda,j}}
		\sum_{\ell=1}^{\infty}
		\hat{\Delta}_j\beta_\ell
		\inpro{\left(\bff m_h^{j+\frac12}\times\Delta_h\bff m_h^{j+1}-\alpha\bff m_h^{j+\frac12}\times\left(\bff m_h^{j+\frac12}\times\Delta_h\bff m_h^{j+1}\right)\right)\times\bff g_{\ell,h}}{\bff\theta_h^{j+1}}_h,
		\label{eq:S5-def-sllg}
		\\
		\mathcal S_r^{(6)}
		&:=
		-\frac12
		\sum_{j=0}^{r-1}
		\mathbf1_{\Omega_{\Lambda,j}}
		\sum_{\ell=1}^{\infty}
		\sum_{\substack{q=1\\q\ne\ell}}^{\infty}
		\hat{\Delta}_j\beta_\ell\,\hat{\Delta}_j\beta_q
		\inpro{\left(\bff m_h^{j+\frac12}\times\bff g_{q,h}\right)\times\bff g_{\ell,h}}{\bff\theta_h^{j+1}}_h.
		\label{eq:S6-def-sllg}
	\end{align}
	
	We estimate these six terms separately.
	For $\mathcal S_r^{(1)}$, write
	$\bff\theta_h^{j+1}=\bff\theta_h^j+
	(\bff\theta_h^{j+1}-\bff\theta_h^j)$. For the part paired with $\bff\theta_h^{j+1}-\bff\theta_h^j$, using Young's inequality followed by the It\^o isometry and \eqref{eq:exact-time-increment-sum-sllg}, we have
	\begin{align*}
		&\mathbb E\left[
		\sup_{0\le r\le J}
		\left|
		\sum_{j=0}^{r-1}
		\mathbf1_{\Omega_{\Lambda,j}}
		\inpro{
			\sum_{\ell=1}^{\infty}
			\int_{t_j}^{t_{j+1}}
			I_h(\bff m(s)-\bff m^j)
			\times\bff g_{\ell,h}\,\dbeta_\ell(s)
		}{
			\bff\theta_h^{j+1}-\bff\theta_h^j
		}_h
		\right|
		\right]
		\nonumber\\
		&\quad\le
		\delta\mathbb E\left[
		\sum_{j=0}^{J-1}
		\mathbf1_{\Omega_{\Lambda,j}}
		\norm{\bff\theta_h^{j+1}-\bff\theta_h^j}{h}^2
		\right]
		+
		C_\delta\sum_{j=0}^{J-1}
		\mathbb E\left[
		\norm{
			\displaystyle
			\sum_{\ell=1}^{\infty}
			\int_{t_j}^{t_{j+1}}
			I_h(\bff m(s)-\bff m^j)
			\times\bff g_{\ell,h}\,\dbeta_\ell(s)
		}{h}^2
		\right]
		\nonumber\\
		&\quad\le
		\delta\mathbb E\left[
		\sum_{j=0}^{J-1}
		\mathbf1_{\Omega_{\Lambda,j}}
		\norm{\bff\theta_h^{j+1}-\bff\theta_h^j}{h}^2
		\right]
		+C_{\delta,\varepsilon}k^{1-\varepsilon}.
	\end{align*}
	The part paired with $\bff\theta_h^j$ is a martingale. By the discrete
	BDG inequality, \eqref{eq:Ih-product-L2-sllg}, and Young's inequality,
	we obtain for every $\eta>0$,
	\begin{align*}
		&\mathbb E\left[
		\sup_{0\le r\le J}
		\left|
		\sum_{j=0}^{r-1}
		\mathbf1_{\Omega_{\Lambda,j}}
		\sum_{\ell=1}^{\infty}
		\int_{t_j}^{t_{j+1}}
		\inpro{
			I_h(\bff m(s)-\bff m^j)\times\bff g_{\ell,h}
		}{
			\bff\theta_h^j
		}_h
		\dbeta_\ell(s)
		\right|
		\right]
		\nonumber\\
		&\qquad\le
		\eta\mathbb E\left[
		\sup_{0\le j\le J}
		\mathbf1_{\Omega_{\Lambda,j}}
		\norm{\bff\theta_h^j}{h}^2
		\right]
		+C_{\eta,\varepsilon}k^{1-\varepsilon}.
	\end{align*}
	Thus, we have for any $\delta,\eta>0$,
	\begin{align}
		\mathbb E\left[
		\sup_{0\le r\le J}|\mathcal S_r^{(1)}|
		\right]
		&\le
		\delta\mathbb E\left[
		\sum_{j=0}^{J-1}
		\mathbf1_{\Omega_{\Lambda,j}}
		\norm{\bff\theta_h^{j+1}-\bff\theta_h^j}{h}^2
		\right]
		+
		\eta\mathbb E\left[
		\sup_{0\le j\le J}
		\mathbf1_{\Omega_{\Lambda,j}}
		\norm{\bff\theta_h^j}{h}^2
		\right]
		+C_{\delta,\eta,\varepsilon}k^{1-\varepsilon}.
		\label{eq:S1-estimate-sllg}
	\end{align}
	
	For $\mathcal S_r^{(2)}$, we note that $\inpro{\bff\theta_h^j\times\hat{\Delta}_jW_h}{\bff\theta_h^{j+1}}_h
	=
	\inpro{\bff\theta_h^j\times\hat{\Delta}_jW_h}{\bff\theta_h^{j+1}-\bff\theta_h^j}_h$.
	Therefore, by Young's inequality and the independence of $\hat{\Delta}_jW_h$ from
	$\mathcal F_{t_j}$,
	\begin{align}
		\mathbb E\left[
		\sup_{0\le r\le J}|\mathcal S_r^{(2)}|
		\right]
		&\le
		\delta\mathbb E\left[
		\sum_{j=0}^{J-1}
		\mathbf1_{\Omega_{\Lambda,j}}
		\norm{\bff\theta_h^{j+1}-\bff\theta_h^j}{h}^2
		\right]
		+
		C_\delta k\sum_{j=0}^{J-1}
		\mathbb E\left[
		\mathbf1_{\Omega_{\Lambda,j}}
		\norm{\bff\theta_h^j}{h}^2
		\right]
		\nonumber\\
		&\leq
		\delta\mathbb E\left[
		\sum_{j=0}^{J-1}
		\mathbf1_{\Omega_{\Lambda,j}}
		\norm{\bff\theta_h^{j+1}-\bff\theta_h^j}{h}^2
		\right]
		+
		C_\delta k\sum_{j=0}^{J-1}
		\mathbb E\left[
		\mathbf1_{\Omega_{\Lambda,j}}
		\norm{\bff\theta_h^{j+1}}{h}^2
		\right],
		\label{eq:S2-estimate-sllg}
	\end{align}
	where in the last step we used the fact that $\bff\theta_h^0=\bff0$ and
	$\Omega_{\Lambda,j}\subset\Omega_{\Lambda,j-1}$ for $j\ge 1$.

	Next, we estimate $\mathcal S_r^{(3)}$. We first record the consequence of
	the product estimate \eqref{eq:Ih-product-L2-sllg} which will be used repeatedly below. For
	$\bff v_h,\bff w_h\in\bb V_h$,
	\begin{align}
		\sum_{\ell=1}^{\infty}
		\left|
		\inpro{
			(\bff v_h\times\bff g_{\ell,h})\times\bff g_{\ell,h}
		}{
			\bff w_h
		}_h
		\right|
		&
		=
		\sum_{\ell=1}^{\infty}
		\left|
		\inpro{
			I_h\left[
			(\bff v_h\times\bff g_{\ell,h})\times\bff g_{\ell,h}
			\right]
		}{
			\bff w_h
		}_h
		\right|
		\nonumber
		\\
		&\le
		C
		\sum_{\ell=1}^{\infty}
		\norm{
			I_h\left[
			(\bff v_h\times\bff g_{\ell,h})\times\bff g_{\ell,h}
			\right]
		}{\bb L^2}
		\norm{\bff w_h}{h}
		\nonumber
		\\
		&
		\le
		C_\Gamma
		\norm{\bff v_h}{h}
		\norm{\bff w_h}{h},
		\label{eq:double-cross-noise-product-bound-sllg}
	\end{align}
	where we also used the discrete norm equivalence \eqref{eq:lumped-L2-equivalence-sllg} and the assumption \eqref{eq:noise-W1inf-l1-assumption-sllg}.
	Now, using
	\[
	I_h\bff m(s)-\bff m_h^{j+\frac12}
	=
	I_h(\bff m(s)-\bff m^{j+1})
	+
	\bff\theta_h^{j+1}
	+
	\frac12(\bff m_h^{j+1}-\bff m_h^j),
	\]
	we estimate the three resulting contributions in $\mathcal S_r^{(3)}$ directly. For the first, estimate~\eqref{eq:double-cross-noise-product-bound-sllg}, Young's inequality,
	the $\bb{L}^\infty$-stability of $I_h$, and \eqref{eq:exact-time-increment-sum-sllg} give
	\begin{align}
		&\mathbb E\left[
		\sup_{0\le r\le J}
		\left|
		\frac12\sum_{j=0}^{r-1}
		\mathbf1_{\Omega_{\Lambda,j}}
		\sum_{\ell=1}^{\infty}
		\int_{t_j}^{t_{j+1}}
		\inpro{
			\left(
			I_h(\bff m(s)-\bff m^{j+1})
			\times\bff g_{\ell,h}
			\right)\times\bff g_{\ell,h}
		}{
			\bff\theta_h^{j+1}
		}_h \ds
		\right|
		\right]
		\nonumber\\
		&\qquad\le
		C_\varepsilon k^{1-\varepsilon}
		+Ck\sum_{j=0}^{J-1}
		\mathbb E\left[
		\mathbf1_{\Omega_{\Lambda,j}}
		\norm{\bff\theta_h^{j+1}}{h}^2
		\right].
		\label{eq:S3-1-estimate-sllg}
	\end{align}
	The contribution containing $\bff\theta_h^{j+1}$ satisfies
	\begin{align*}
		&\mathbb E\left[
		\sup_{0\le r\le J}
		\left|
		\frac12\sum_{j=0}^{r-1}
		\mathbf1_{\Omega_{\Lambda,j}}
		\sum_{\ell=1}^{\infty}
		\int_{t_j}^{t_{j+1}}
		\inpro{
			(\bff\theta_h^{j+1}\times\bff g_{\ell,h})
			\times\bff g_{\ell,h}
		}{
			\bff\theta_h^{j+1}
		}_h\ds
		\right|
		\right]
		\le
		Ck\sum_{j=0}^{J-1}
		\mathbb E\left[
		\mathbf1_{\Omega_{\Lambda,j}}
		\norm{\bff\theta_h^{j+1}}{h}^2
		\right].
	\end{align*}
	Finally, \eqref{eq:double-cross-noise-product-bound-sllg}, Young's
	inequality, and the increment moment bound~\eqref{eq:mh-increment-second-moment-sllg} imply
	\begin{align}
		&\mathbb E\left[
		\sup_{0\le r\le J}
		\left|
		\frac14\sum_{j=0}^{r-1}
		\mathbf1_{\Omega_{\Lambda,j}}
		\sum_{\ell=1}^{\infty}
		\int_{t_j}^{t_{j+1}}
		\inpro{
			\left(
			(\bff m_h^{j+1}-\bff m_h^j)
			\times\bff g_{\ell,h}
			\right)\times\bff g_{\ell,h}
		}{
			\bff\theta_h^{j+1}
		}_h \ds
		\right|
		\right]
		\nonumber\\
		&\qquad\le
		Ck\sum_{j=0}^{J-1}
		\mathbb E\left[
		\norm{\bff m_h^{j+1}-\bff m_h^j}{h}^2
		\right]
		+
		Ck\sum_{j=0}^{J-1}
		\mathbb E\left[
		\mathbf1_{\Omega_{\Lambda,j}}
		\norm{\bff\theta_h^{j+1}}{h}^2
		\right]
		\nonumber\\
		&\qquad\le
		Ck
		+Ck\sum_{j=0}^{J-1}
		\mathbb E\left[
		\mathbf1_{\Omega_{\Lambda,j}}
		\norm{\bff\theta_h^{j+1}}{h}^2
		\right].
		\label{eq:S3-3-preliminary-sllg}
	\end{align}
	Combining \eqref{eq:S3-1-estimate-sllg}--\eqref{eq:S3-3-preliminary-sllg} gives
	\begin{align}
		\mathbb E\left[
		\sup_{0\le r\le J}|\mathcal S_r^{(3)}|
		\right]
		&\le
		C_\varepsilon k^{1-\varepsilon}
		+Ck\sum_{j=0}^{J-1}
		\mathbb E\left[
		\mathbf1_{\Omega_{\Lambda,j}}
		\norm{\bff\theta_h^{j+1}}{h}^2
		\right].
		\label{eq:S3-estimate-sllg}
	\end{align}
	
	For the term $\mathcal S_r^{(4)}$, we decompose
	\[
	\bff\theta_h^{j+1}
	=
	\bff\theta_h^j
	+
	(\bff\theta_h^{j+1}-\bff\theta_h^j),
	\qquad
	\bff m_h^{j+\frac12}
	=
	\bff m_h^j
	+
	\frac12(\bff m_h^{j+1}-\bff m_h^j).
	\]
	We first estimate the part paired with
	$\bff\theta_h^{j+1}-\bff\theta_h^j$. By Young's inequality, for any $\delta>0$,
	\begin{align}
		&\mathbf1_{\Omega_{\Lambda,j}}
		\left|
		\sum_{\ell=1}^{\infty}
		\left(|\hat{\Delta}_j\beta_\ell|^2-k\right)
		\inpro{
			\left(\bff m_h^{j+\frac12}\times\bff g_{\ell,h}\right)\times\bff g_{\ell,h}
		}{
			\bff\theta_h^{j+1}-\bff\theta_h^j
		}_h
		\right|
		\nonumber
		\\
		&\quad
		\le
		\delta\,
		\mathbf1_{\Omega_{\Lambda,j}}
		\norm{\bff\theta_h^{j+1}-\bff\theta_h^j}{h}^2
		+
		C_\delta
		\left(
		\sum_{\ell=1}^{\infty}
		\left||\hat{\Delta}_j\beta_\ell|^2-k\right|
		\norm{\bff g_{\ell,h}}{\bb L^\infty}^2
		\right)^2
		\norm{\bff m_h^{j+\frac12}}{h}^2 .
		\label{eq:S4A-young-sllg}
	\end{align}
	We note that by the assumption \eqref{eq:noise-W1inf-l1-assumption-sllg} and the fact that $\mathbb E\left[\left||\hat{\Delta}_j\beta_\ell|^2-k\right|^2\right]= 2k^2$, we have
	\begin{equation}\label{eq:Delta-beta-g}
		\mathbb E
		\left[
		\left(
		\sum_{\ell=1}^{\infty}
		\left||\hat{\Delta}_j\beta_\ell|^2-k\right|
		\norm{\bff g_{\ell,h}}{\bb L^\infty}^2
		\right)^2
		\right]
		\le
		Ck^2.
	\end{equation}
	Therefore, noting \eqref{eq:Linfty-stability-sllg}, summing \eqref{eq:S4A-young-sllg} over $j$ and taking expectations, we obtain
	\begin{align}
		&\mathbb E
		\left[
		\sum_{j=0}^{J-1}
		\mathbf1_{\Omega_{\Lambda,j}}
		\left|
		\sum_{\ell=1}^{\infty}
		\left(|\hat{\Delta}_j\beta_\ell|^2-k\right)
		\inpro{
			\left(\bff m_h^{j+\frac12}\times\bff g_{\ell,h}\right)\times\bff g_{\ell,h}
		}{
			\bff\theta_h^{j+1}-\bff\theta_h^j
		}_h
		\right|
		\right]
		\nonumber
		\\
		&\quad
		\le
		\delta
		\mathbb E
		\left[
		\sum_{j=0}^{J-1}
		\mathbf1_{\Omega_{\Lambda,j}}
		\norm{\bff\theta_h^{j+1}-\bff\theta_h^j}{h}^2
		\right]
		+
		C_\delta k .
		\label{eq:S4A-estimate-sllg}
	\end{align}
	Hence, thanks to the discrete BDG inequality
	(Lemma~\ref{lem:discrete-BDG} with $H=\mathbb R$ and $p=2$), we have
	\begin{align}
		&\mathbb E
		\left[
		\sup_{0\le r\le J}
		\left|
		\sum_{j=0}^{r-1}
		\mathbf1_{\Omega_{\Lambda,j}}
		\sum_{\ell=1}^{\infty}
		\left(|\hat{\Delta}_j\beta_\ell|^2-k\right)
		\inpro{
			\left(\bff m_h^j\times\bff g_{\ell,h}\right)\times\bff g_{\ell,h}
		}{
			\bff\theta_h^j
		}_h
		\right|
		\right]
		\nonumber
		\\
		&\quad
		\le
		C
		\left(
		\sum_{j=0}^{J-1}
		\mathbb E
		\left[
		\mathbf1_{\Omega_{\Lambda,j}}
		\left|
		\sum_{\ell=1}^{\infty}
		\left(|\hat{\Delta}_j\beta_\ell|^2-k\right)
		\inpro{
			\left(\bff m_h^j\times\bff g_{\ell,h}\right)\times\bff g_{\ell,h}
		}{
			\bff\theta_h^j
		}_h
		\right|^2
		\right]
		\right)^{\frac12}
		\nonumber
		\\
		&\quad \leq
		C \left( \sum_{j=0}^{J-1}
		\bb{E} \left[ \mathbf1_{\Omega_{\Lambda,j}} \norm{\bff\theta_h^j}{h}^2
		\left( \sum_{\ell=1}^{\infty}
		\left||\hat{\Delta}_j\beta_\ell|^2-k\right|
		\norm{\bff g_{\ell,h}}{\bb L^\infty}^2 \right)^2 \right] \right)^{\frac12}
		\nonumber
		\\
		&\quad=
		C \left( \sum_{j=0}^{J-1}
		\bb{E} \left[ \mathbf1_{\Omega_{\Lambda,j}} \norm{\bff\theta_h^j}{h}^2 \right] 
		\bb{E} \left[
		\left( \sum_{\ell=1}^{\infty}
		\left||\hat{\Delta}_j\beta_\ell|^2-k\right|
		\norm{\bff g_{\ell,h}}{\bb L^\infty}^2 \right)^2 \right] \right)^{\frac12},
		\label{eq:S4B-bdg-sllg}
	\end{align}
	where in the penultimate step we used the Cauchy--Schwarz inequality, \eqref{eq:Linfty-stability-sllg}, and \eqref{eq:noise-W1inf-l1-assumption-sllg}, while in the last step we used the fact that $\mathbf1_{\Omega_{\Lambda,j}}\norm{\bff\theta_h^j}{h}^2$
	is $\mathcal F_{t_j}$-measurable and the Brownian increments over
	$[t_j,t_{j+1}]$ are independent of $\mathcal F_{t_j}$.
	Therefore, continuing from \eqref{eq:S4B-bdg-sllg}, we obtain by Young's inequality and \eqref{eq:Delta-beta-g} that
	\begin{align}
		&\mathbb E
		\left[
		\sup_{0\le r\le J}
		\left|
		\sum_{j=0}^{r-1}
		\mathbf1_{\Omega_{\Lambda,j}}
		\sum_{\ell=1}^{\infty}
		\left(|\hat{\Delta}_j\beta_\ell|^2-k\right)
		\inpro{
			\left(\bff m_h^j\times\bff g_{\ell,h}\right)\times\bff g_{\ell,h}
		}{
			\bff\theta_h^j
		}_h
		\right|
		\right]
		\nonumber \\
		&\quad
		\le
		\eta
		\mathbb E
		\left[
		\max_{0\le j\le J}
		\mathbf1_{\Omega_{\Lambda,j}}
		\norm{\bff\theta_h^j}{h}^2
		\right]
		+
		C_\eta k .
		\label{eq:S4B-estimate-sllg}
	\end{align}
	It remains to estimate the part containing
	$\bff m_h^{j+1}-\bff m_h^j$. By
	\eqref{eq:double-cross-noise-product-bound-sllg} and Young's inequality,
	\begin{align}
		&\mathbb E
		\left[
		\sum_{j=0}^{J-1}
		\mathbf1_{\Omega_{\Lambda,j}}
		\left|
		\sum_{\ell=1}^{\infty}
		\left(|\hat{\Delta}_j\beta_\ell|^2-k\right)
		\inpro{
			\left((\bff m_h^{j+1}-\bff m_h^j)\times\bff g_{\ell,h}\right)\times\bff g_{\ell,h}
		}{
			\bff\theta_h^j
		}_h
		\right|
		\right]
		\nonumber
		\\
		&\quad
		\le
		C
		\sum_{j=0}^{J-1}
		\mathbb E
		\left[
		\mathbf1_{\Omega_{\Lambda,j}}
		\left(
		\sum_{\ell=1}^{\infty}
		\left||\hat{\Delta}_j\beta_\ell|^2-k\right|
		\norm{\bff g_{\ell,h}}{\bb L^\infty}^2
		\right)
		\norm{\bff m_h^{j+1}-\bff m_h^j}{h}
		\norm{\bff\theta_h^j}{h}
		\right]
		\nonumber
		\\
		&\quad
		\le
		C
		\sum_{j=0}^{J-1}
		\mathbb E
		\left[
		\left(
		\sum_{\ell=1}^{\infty}
		\left||\hat{\Delta}_j\beta_\ell|^2-k\right|
		\norm{\bff g_{\ell,h}}{\bb L^\infty}^2
		\right)
		\norm{\bff m_h^{j+1}-\bff m_h^j}{h}^2
		\right]
		\nonumber
		\\
		&\qquad
		+
		C
		\sum_{j=0}^{J-1}
		\mathbb E
		\left[
		\mathbf1_{\Omega_{\Lambda,j}}
		\left(
		\sum_{\ell=1}^{\infty}
		\left||\hat{\Delta}_j\beta_\ell|^2-k\right|
		\norm{\bff g_{\ell,h}}{\bb L^\infty}^2
		\right)
		\norm{\bff\theta_h^j}{h}^2
		\right].
		\label{eq:S4C-young-sllg}
	\end{align}
	For the first term on the right-hand side, Hölder's inequality and \eqref{eq:Delta-beta-g} yield
	\begin{align}
		\sum_{j=0}^{J-1}
		\mathbb E
		\left[
		\left(
		\sum_{\ell=1}^{\infty}
		\left||\hat{\Delta}_j\beta_\ell|^2-k\right|
		\norm{\bff g_{\ell,h}}{\bb L^\infty}^2
		\right)
		\norm{\bff m_h^{j+1}-\bff m_h^j}{h}^2
		\right]
		&
		\le
		Ck
		\sum_{j=0}^{J-1}
		\left(
		\mathbb E
		\left[
		\norm{\bff m_h^{j+1}-\bff m_h^j}{h}^4
		\right]
		\right)^{\frac12}
		\nonumber
		\\
		&
		\le
		Ck,
		\label{eq:S4C-increment-part-sllg}
	\end{align}
	where in the final step we also used the increment moment bound \eqref{eq:mh-increment-fourth-root-sllg}.
	For the last term on the right-hand side of
	\eqref{eq:S4C-young-sllg}, the factor involving the Brownian increment is
	independent of $\mathcal F_{t_j}$ and has expectation bounded by $Ck$.
	Therefore,
	\begin{align}
		&\sum_{j=0}^{J-1}
		\mathbb E
		\left[
		\mathbf1_{\Omega_{\Lambda,j}}
		\left(
		\sum_{\ell=1}^{\infty}
		\left||\hat{\Delta}_j\beta_\ell|^2-k\right|
		\norm{\bff g_{\ell,h}}{\bb L^\infty}^2
		\right)
		\norm{\bff\theta_h^j}{h}^2
		\right]
		\le
		Ck
		\sum_{j=0}^{J-1}
		\mathbb E
		\left[
		\mathbf1_{\Omega_{\Lambda,j}}
		\norm{\bff\theta_h^j}{h}^2
		\right].
		\label{eq:S4C-theta-part-sllg}
	\end{align}
	Combining \eqref{eq:S4A-estimate-sllg}, \eqref{eq:S4B-estimate-sllg}, \eqref{eq:S4C-young-sllg}, \eqref{eq:S4C-increment-part-sllg}, and
	\eqref{eq:S4C-theta-part-sllg}, and absorbing harmless numerical constants into
	$\delta$ and $\eta$, we obtain
	\begin{align}
		\mathbb E
		\left[
		\sup_{0\le r\le J}
		|\mathcal S_r^{(4)}|
		\right]
		&\le
		\delta
		\mathbb E
		\left[
		\sum_{j=0}^{J-1}
		\mathbf1_{\Omega_{\Lambda,j}}
		\norm{\bff\theta_h^{j+1}-\bff\theta_h^j}{h}^2
		\right]
		+
		\eta
		\mathbb E
		\left[
		\sup_{0\le j\le J}
		\mathbf1_{\Omega_{\Lambda,j}}
		\norm{\bff\theta_h^j}{h}^2
		\right]
		\nonumber
		\\
		&\qquad
		+
		C_{\delta,\eta}k
		+
		Ck
		\sum_{j=0}^{J-1}
		\mathbb E
		\left[
		\mathbf1_{\Omega_{\Lambda,j}}
		\norm{\bff\theta_h^{j+1}}{h}^2
		\right].
		\label{eq:S4-estimate-sllg}
	\end{align}
	
	We estimate $\mathcal S_r^{(5)}$ next. By the mass-lumped Cauchy--Schwarz
	inequality and the nodal bound \eqref{eq:nodal-constraint-sllg}, we first have
	\begin{align}
		|\mathcal S_r^{(5)}|
		&\le
		Ck
		\sum_{j=0}^{r-1}
		\mathbf1_{\Omega_{\Lambda,j}}
		\norm{
			\bff m_h^{j+\frac12}\times\Delta_h\bff m_h^{j+1}
		}{h}
		\norm{\hat{\Delta}_jW_h}{\bb L^\infty}
		\norm{\bff\theta_h^{j+1}}{h}.
		\label{eq:S5-start-sllg}
	\end{align}
	Now, let $\rho>0$. Young's inequality with conjugate exponents
	$(2+\rho)/(1+\rho)$ and $2+\rho$, together with the uniform bound
	$\norm{\bff\theta_h^{j+1}}{h}\le C$, gives
	\begin{align}
		|\mathcal S_r^{(5)}|
		&\le
		C_\rho k
		\sum_{j=0}^{J-1}
		\norm{
			\bff m_h^{j+\frac12}\times\Delta_h\bff m_h^{j+1}
		}{h}^{\frac{2+\rho}{1+\rho}}
		\norm{\hat{\Delta}_jW_h}{\bb L^\infty}^{\frac{2+\rho}{1+\rho}}
		+
		C_\rho k
		\sum_{j=0}^{J-1}
		\mathbf1_{\Omega_{\Lambda,j}}
		\norm{\bff\theta_h^{j+1}}{h}^{2}.
		\label{eq:S5-young-sllg}
	\end{align}
	By H\"older's inequality and
	\eqref{eq:noise-increment-moment-sllg},
	\begin{align*}
		&k\sum_{j=0}^{J-1}\mathbb E\left[
		\norm{
			\bff m_h^{j+\frac12}\times\Delta_h\bff m_h^{j+1}
		}{h}^{\frac{2+\rho}{1+\rho}}
		\norm{\hat{\Delta}_jW_h}{\bb L^\infty}^{\frac{2+\rho}{1+\rho}}
		\right]
		\le
		C_\rho k^{1+\frac{2+\rho}{2+2\rho}}
		\sum_{j=0}^{J-1}
		\left(
		\mathbb E\left[
		\norm{
			\bff m_h^{j+\frac12}\times\Delta_h\bff m_h^{j+1}
		}{h}^{2}
		\right]
		\right)^{\frac{2+\rho}{2+2\rho}}.
	\end{align*}
	Indeed, the moment of the Wiener increment used here is of order
	$2+4/\rho$. Since $\frac{2+\rho}{2+2\rho} \in (0,1)$, the discrete Jensen
	inequality (Lemma~\ref{lem:weighted-discrete-holder}) and Proposition~\ref{prop:p-moment-stability-sllg} yield
	\begin{align*}
		&\sum_{j=0}^{J-1}
		\left(
		\mathbb E\left[
		\norm{
			\bff m_h^{j+\frac12}\times\Delta_h\bff m_h^{j+1}
		}{h}^{2}
		\right]
		\right)^{\frac{2+\rho}{2+2\rho}}
		\le
		J^{\frac{\rho}{2+2\rho}}
		\left(
		\sum_{j=0}^{J-1}
		\mathbb E\left[
		\norm{
			\bff m_h^{j+\frac12}\times\Delta_h\bff m_h^{j+1}
		}{h}^{2}
		\right]
		\right)^{\frac{2+\rho}{2+2\rho}}
		\le C_\rho k^{-1}.
	\end{align*}
	This estimate, together with \eqref{eq:S5-young-sllg}, implies
	\begin{align}
		\mathbb E\left[
		\sup_{0\le r\le J}|\mathcal S_r^{(5)}|
		\right]
		&\le
		C_\rho k^{\frac{2+\rho}{2+2\rho}}
		+C_\rho k
		\sum_{j=0}^{J-1}
		\mathbb E\left[
		\mathbf1_{\Omega_{\Lambda,j}}
		\norm{\bff\theta_h^{j+1}}{h}^{2}
		\right]
		\nonumber\\
		&\le
		C_\varepsilon k^{1-\varepsilon}
		+C_\varepsilon k
		\sum_{j=0}^{J-1}
		\mathbb E\left[
		\mathbf1_{\Omega_{\Lambda,j}}
		\norm{\bff\theta_h^{j+1}}{h}^{2}
		\right],
		\label{eq:S5-estimate-sllg}
	\end{align}
	where $\rho>0$ is chosen so that
	$\rho/(2+2\rho)\le\varepsilon$.
	
	Finally, we estimate $\mathcal S_r^{(6)}$. To avoid repeatedly
	writing the same off-diagonal coefficient, define
	\[
	\mathfrak X_j
	:=
	\sum_{\ell=1}^{\infty}
	\sum_{\substack{q=1\\q\ne\ell}}^{\infty}
	|\hat{\Delta}_j\beta_\ell\hat{\Delta}_j\beta_q|
	\norm{\bff g_{\ell,h}}{\bb L^\infty}
	\norm{\bff g_{q,h}}{\bb L^\infty}.
	\]
	By Gaussian moments and independence of Brownian increments, Minkowski's inequality in $L^2(\Omega)$, as well as assumption~\eqref{eq:noise-W1inf-l1-assumption-sllg}, we have
	\begin{equation}\label{eq:off-diagonal-increment-moments-sllg}
		\mathbb E[\mathfrak X_j]\le Ck,
		\qquad
		\mathbb E[\mathfrak X_j^2]\le Ck^2.
	\end{equation}
	By writing
	\[
	\bff m_h^{j+\frac12}
	=\bff m_h^j
	+\frac12(\bff m_h^{j+1}-\bff m_h^j) 
	\;\text{ and }\; 
	\bff\theta_h^{j+1}
	=\bff\theta_h^j
	+(\bff\theta_h^{j+1}-\bff\theta_h^j),
	\]
	we first estimate the part paired with the error increment $\bff\theta_h^{j+1}-\bff\theta_h^j$. The nodal
	bound \eqref{eq:nodal-constraint-sllg} and Young's inequality imply
	\begin{align*}
		&\mathbb E\left[
		\sup_{0\le r\le J}
		\left|
		\sum_{j=0}^{r-1}\mathbf1_{\Omega_{\Lambda,j}}
		\sum_{\substack{\ell,q\ge1\\q\ne\ell}}
		\hat{\Delta}_j\beta_\ell\hat{\Delta}_j\beta_q
		\inpro{
			(\bff m_h^{j+\frac12}\times\bff g_{q,h})
			\times\bff g_{\ell,h}
		}{
			\bff\theta_h^{j+1}-\bff\theta_h^j
		}_h
		\right|
		\right]
		\nonumber\\
		&\qquad\le
		\delta\mathbb E\left[
		\sum_{j=0}^{J-1}\mathbf1_{\Omega_{\Lambda,j}}
		\norm{\bff\theta_h^{j+1}-\bff\theta_h^j}{h}^2
		\right]
		+C_\delta k.
	\end{align*}
	The leading part paired with $\bff\theta_h^j$ is a martingale because,
	for $q\ne\ell$, we have $\mathbb E\left[
	\hat{\Delta}_j\beta_\ell\hat{\Delta}_j\beta_q
	\mid\mathcal F_{t_j}
	\right]=0$.
	Moreover, for every scalar array
	$\{c_{\ell q}\}_{\ell\ne q}$,
	\[
	\begin{aligned}
		&\mathbb E\left[
		\left|
		\sum_{\substack{\ell,q\ge1\\q\ne\ell}}
		c_{\ell q}
		\hat{\Delta}_j\beta_\ell
		\hat{\Delta}_j\beta_q
		\right|^2
		\right]
		\le
		2k^2
		\left(
		\sum_{\substack{\ell,q\ge1\\q\ne\ell}}
		|c_{\ell q}|
		\right)^2.
	\end{aligned}
	\]
	Applying this conditionally on $\mathcal F_{t_j}$, followed by the
	discrete BDG and Young inequalities, yields
	\begin{align*}
		&\mathbb E\left[
		\sup_{0\le r\le J}
		\left|
		\sum_{j=0}^{r-1}
		\mathbf1_{\Omega_{\Lambda,j}}
		\sum_{\substack{\ell,q\ge1\\q\ne\ell}}
		\hat{\Delta}_j\beta_\ell
		\hat{\Delta}_j\beta_q
		\inpro{
			(\bff m_h^j\times\bff g_{q,h})
			\times\bff g_{\ell,h}
		}{
			\bff\theta_h^j
		}_h
		\right|
		\right]
		\le
		\eta\mathbb E\left[
		\sup_{0\le j\le J}
		\mathbf1_{\Omega_{\Lambda,j}}
		\norm{\bff\theta_h^j}{h}^2
		\right]
		+C_\eta k.
	\end{align*}
	It remains to estimate the contribution arising from
	$\frac12(\bff m_h^{j+1}-\bff m_h^j)$ paired with $\bff\theta_h^j$. Using the
	definition of $\mathfrak X_j$ and the mass-lumped Cauchy--Schwarz
	inequality, we obtain
	\begin{align*}
		&\mathbb E\left[
		\sup_{0\le r\le J}
		\left|
		\sum_{j=0}^{r-1}
		\mathbf1_{\Omega_{\Lambda,j}}
		\sum_{\substack{\ell,q\ge1\\q\ne\ell}}
		\hat{\Delta}_j\beta_\ell\hat{\Delta}_j\beta_q
		\inpro{
			\left(
			(\bff m_h^{j+1}-\bff m_h^j)\times\bff g_{q,h}
			\right)\times\bff g_{\ell,h}
		}{
			\bff\theta_h^j
		}_h
		\right|
		\right]
		\nonumber\\
		&\qquad
		\le
		C\sum_{j=0}^{J-1}
		\mathbb E\left[
		\mathbf1_{\Omega_{\Lambda,j}}
		\mathfrak X_j
		\norm{\bff m_h^{j+1}-\bff m_h^j}{h}
		\norm{\bff\theta_h^j}{h}
		\right]
		\\
		&\qquad
		\le
		C\sum_{j=0}^{J-1}
		\mathbb E\left[
		\mathfrak X_j
		\norm{\bff m_h^{j+1}-\bff m_h^j}{h}^2
		\right]
		+
		C\sum_{j=0}^{J-1}
		\mathbb E\left[
		\mathbf1_{\Omega_{\Lambda,j}}
		\mathfrak X_j
		\norm{\bff\theta_h^j}{h}^2
		\right].
	\end{align*}
	
	For the first term on the right-hand side, applying the Cauchy--Schwarz inequality, \eqref{eq:off-diagonal-increment-moments-sllg}, and
	\eqref{eq:mh-increment-fourth-root-sllg} give
	\begin{align*}
		\sum_{j=0}^{J-1}
		\mathbb E\left[
		\mathfrak X_j
		\norm{\bff m_h^{j+1}-\bff m_h^j}{h}^2
		\right]
		&\le
		\sum_{j=0}^{J-1}
		\left(
		\mathbb E[\mathfrak X_j^2]
		\right)^{\frac12}
		\left(
		\mathbb E\left[
		\norm{\bff m_h^{j+1}-\bff m_h^j}{h}^4
		\right]
		\right)^{\frac12}
		\nonumber\\
		&\le
		Ck
		\sum_{j=0}^{J-1}
		\left(
		\mathbb E\left[
		\norm{\bff m_h^{j+1}-\bff m_h^j}{h}^4
		\right]
		\right)^{\frac12}
		\le
		Ck.
	\end{align*}
	
	For the second term, both
	$\mathbf1_{\Omega_{\Lambda,j}}$ and $\bff\theta_h^j$ are
	$\mathcal F_{t_j}$-measurable, whereas $\mathfrak X_j$ depends only
	on the Brownian increments over $[t_j,t_{j+1}]$ and is therefore
	independent of $\mathcal F_{t_j}$. Hence, again by
	\eqref{eq:off-diagonal-increment-moments-sllg},
	\begin{align*}
		\sum_{j=0}^{J-1}
		\mathbb E\left[
		\mathbf1_{\Omega_{\Lambda,j}}
		\mathfrak X_j
		\norm{\bff\theta_h^j}{h}^2
		\right]
		&=
		\sum_{j=0}^{J-1}
		\mathbb E[\mathfrak X_j]\,
		\mathbb E\left[
		\mathbf1_{\Omega_{\Lambda,j}}
		\norm{\bff\theta_h^j}{h}^2
		\right]
		\le
		Ck\sum_{j=0}^{J-1}
		\mathbb E\left[
		\mathbf1_{\Omega_{\Lambda,j}}
		\norm{\bff\theta_h^{j+1}}{h}^2
		\right],
	\end{align*}
	where the last step follows from the fact that $\bff\theta_h^0=\bff0$ and
	$\Omega_{\Lambda,j}\subset\Omega_{\Lambda,j-1}$ for $j\ge 1$.
	We now combine the above estimates to obtain
	\begin{align}
		\mathbb E\left[
		\sup_{0\le r\le J}|\mathcal S_r^{(6)}|
		\right]
		&\le
		\delta\mathbb E\left[
		\sum_{j=0}^{J-1}
		\mathbf1_{\Omega_{\Lambda,j}}
		\norm{\bff\theta_h^{j+1}-\bff\theta_h^j}{h}^2
		\right]
		+
		\eta\mathbb E\left[
		\sup_{0\le j\le J}
		\mathbf1_{\Omega_{\Lambda,j}}
		\norm{\bff\theta_h^j}{h}^2
		\right]
		+
		C_{\delta,\eta,\varepsilon}k^{1-\varepsilon}
		\nonumber\\
		&\quad+
		Ck\sum_{j=0}^{J-1}
		\mathbb E\left[
		\mathbf1_{\Omega_{\Lambda,j}}
		\norm{\bff\theta_h^{j+1}}{h}^2
		\right].
		\label{eq:S6-estimate-sllg}
	\end{align}
	
	Combining \eqref{eq:S1-estimate-sllg}, \eqref{eq:S2-estimate-sllg},
	\eqref{eq:S3-estimate-sllg}, \eqref{eq:S4-estimate-sllg},
	\eqref{eq:S5-estimate-sllg}, and \eqref{eq:S6-estimate-sllg} completes the proof of \eqref{eq:stochastic-error-term-estimate-sllg}.
\end{proof}

For $\bff a_h,\bff w_h,\bff\phi_h\in\bb V_h$, define the discrete
Landau--Lifshitz drift form
\begin{equation}\label{eq:discrete-drift-form-sllg}
		\mathscr A_h(\bff a_h;\bff w_h,\bff\phi_h)
		\coloneq
		\alpha\kappa
		\inpro{
			\bff a_h\times
			\left(\bff a_h\times\Delta_h\bff w_h\right)
		}{\bff\phi_h}_h
		-\kappa
		\inpro{
			\bff a_h\times\Delta_h\bff w_h
		}{\bff\phi_h}_h.
\end{equation}
We are now in a position to prove our main theorem.

\begin{theorem}\label{thm:localised-error-sllg}
	Let $\alpha,\kappa>0$ and $\gamma\in(0,\frac12)$. Let $\bff m$ be the
	pathwise strong solution of \eqref{eq:sllg-model}, and let
	$\{\bff m_h^j\}_{j=0}^J$ be an adapted solution of
	\eqref{eq:fully-discrete-midpoint-sllg-laplacian}. For every fixed $\Lambda>0$, there exist $k_\Lambda\in (0,1]$ and $C_{\Lambda,\gamma}>0$, independent of $h$ and $k$, such that, whenever $0<k\le k_\Lambda$,
	\begin{align}
		&\mathbb E\left[
		\mathbf1_{\Omega_{\Lambda,J}}
		\max_{0\le j\le J}\norm{\bff\theta_h^j}{h}^2
		\right]
		+\mathbb E\left[
		\mathbf1_{\Omega_{\Lambda,J}}
		\sum_{j=0}^{J-1}
		\norm{\bff\theta_h^{j+1}-\bff\theta_h^j}{h}^2
		\right]
		+
		\mathbb E\left[
		\mathbf1_{\Omega_{\Lambda,J}}
		k\sum_{j=0}^{J-1}
		\norm{\partial_x\bff\theta_h^{j+1}}{\bb L^2}^2
		\right]
		\le
		C_{\Lambda,\gamma}\left(h^2+k^{2\gamma}\right).
		\label{eq:theta-error-local-sllg}
	\end{align}
	Consequently,
	\begin{align}
		&\mathbb E\left[
		\mathbf1_{\Omega_{\Lambda,J}}
		\max_{0\le j\le J}
		\norm{\bff m^j-\bff m_h^j}{\bb L^2}^2
		\right]
		+
		\mathbb E\left[
		\mathbf1_{\Omega_{\Lambda,J}}
		k\sum_{j=0}^{J-1}
		\norm{
			\partial_x(\bff m^{j+1}-\bff m_h^{j+1})
		}{\bb L^2}^2
		\right]
		\le
		C_{\Lambda,\gamma}\left(h^2+k^{2\gamma}\right).
		\label{eq:true-error-local-sllg}
	\end{align}
\end{theorem}

\begin{proof}
	Since $\bff m_h^0=I_h\bff m_0$, we have
	$\bff\theta_h^0=\bff0$. Subtracting
	\eqref{eq:fully-discrete-midpoint-sllg-laplacian} from
	\eqref{eq:projected-exact-discrete-equation-sllg} yields, for every
	$\bff\phi_h\in\bb V_h$,
	\begin{align}
		&\inpro{
			\bff\theta_h^{j+1}-\bff\theta_h^j
		}{\bff\phi_h}_h
		+k\Bigl[
		\mathscr A_h(
		I_h\overline{\bff m}^{j+\frac12};
		I_h\bff m^{j+1},\bff\phi_h
		)
		-
		\mathscr A_h(
		\bff m_h^{j+\frac12};
		\bff m_h^{j+1},\bff\phi_h
		)
		\Bigr]
		\nonumber\\
		&\quad-
		\inpro{
			\displaystyle
			\int_{t_j}^{t_{j+1}}
			I_h\bff m(s)\times\circ\dW_h(s)
			-\bff m_h^{j+\frac12}\times\hat{\Delta}_jW_h
		}{\bff\phi_h}_h
		=
		\mathcal R_h^{j+1}(\bff\phi_h),
		\label{eq:theta-error-equation-sllg}
	\end{align}
	where $\mathscr{A}_h$ is given by \eqref{eq:discrete-drift-form-sllg}.
	Now, choose $\bff\phi_h=\mathbf1_{\Omega_{\Lambda,j}}\bff\theta_h^{j+1}$.
	By Lemma~\ref{lem:local-drift-summable-defect-sllg}, we have
	\begin{align}
		&\mathbf1_{\Omega_{\Lambda,j}}
		\inpro{
			\bff\theta_h^{j+1}-\bff\theta_h^j
		}{\bff\theta_h^{j+1}}_h
		+c_0k\mathbf1_{\Omega_{\Lambda,j}}
		\norm{\partial_x\bff\theta_h^{j+1}}{\bb L^2}^2
		\nonumber\\
		&\quad\le
		C_\Lambda k\mathbf1_{\Omega_{\Lambda,j}}
		\norm{\bff\theta_h^{j+1}}{h}^2
		+
		C_\Lambda k
		\sum_{\nu=1}^{3}
		\left(
		\norm{I_h(\bff m^{j+1}-\bff m^j)}{1,h}^{2\nu}
		+\norm{\bff m_h^{j+1}-\bff m_h^j}{1,h}^{2\nu}
		\right)
		\nonumber\\
		&\qquad+
		\mathbf1_{\Omega_{\Lambda,j}}
		\mathcal R_h^{j+1}(\bff\theta_h^{j+1})
		+
		\mathbf1_{\Omega_{\Lambda,j}}
		\inpro{
			\displaystyle
			\int_{t_j}^{t_{j+1}}
			I_h\bff m(s)\times\circ\dW_h(s)
			-\bff m_h^{j+\frac12}\times\hat{\Delta}_jW_h
		}{\bff\theta_h^{j+1}}_h.
		\label{eq:error-after-corrected-drift-sllg}
	\end{align}
	By the definition of the discrete dual norm and Young's inequality,
	\begin{align*}
		&\mathbf1_{\Omega_{\Lambda,j}}
		\left|
		\mathcal R_h^{j+1}(\bff\theta_h^{j+1})
		\right|
		\le
		\frac{c_0}{4}k\mathbf1_{\Omega_{\Lambda,j}}
		\norm{\partial_x\bff\theta_h^{j+1}}{\bb L^2}^2
		+Ck\mathbf1_{\Omega_{\Lambda,j}}
		\norm{\bff\theta_h^{j+1}}{h}^2
		+\frac{C}{k}\norm{\mathcal R_h^{j+1}}{-1,h}^2.
	\end{align*}
	After absorbing the first term on the right and summing \eqref{eq:error-after-corrected-drift-sllg} from $j=0$ to
	$r-1$, we obtain
	\begin{align}
		&\sum_{j=0}^{r-1}
		\mathbf1_{\Omega_{\Lambda,j}}
		\inpro{
			\bff\theta_h^{j+1}-\bff\theta_h^j
		}{\bff\theta_h^{j+1}}_h
		+c_1k\sum_{j=0}^{r-1}
		\mathbf1_{\Omega_{\Lambda,j}}
		\norm{\partial_x\bff\theta_h^{j+1}}{\bb L^2}^2
		\nonumber\\
		&\quad\le
		C_\Lambda k\sum_{j=0}^{r-1}
		\mathbf1_{\Omega_{\Lambda,j}}
		\norm{\bff\theta_h^{j+1}}{h}^2
		+
		C_\Lambda k\sum_{j=0}^{r-1}
		\sum_{\nu=1}^{3}
		\left(
		\norm{I_h(\bff m^{j+1}-\bff m^j)}{1,h}^{2\nu}
		+\norm{\bff m_h^{j+1}-\bff m_h^j}{1,h}^{2\nu}
		\right)
		\nonumber\\
		&\quad \quad+
		C\sum_{j=0}^{r-1}
		\frac1k\norm{\mathcal R_h^{j+1}}{-1,h}^2
		+\mathcal S_r,
		\label{eq:error-summed-corrected-drift-sllg}
	\end{align}
	where $c_1>0$ and $\mathcal S_r$ is defined by
	\eqref{eq:stochastic-error-term-definition-sllg}.
	
	By Lemma~\ref{lem:stopped-time-identity-sllg}, noting the monotonicity in $r$ of the gradient sum, we obtain for every $1\le n\le J$,
	\begin{align}
		&\max_{1\le r\le n}
		\Bigg[
		\sum_{j=0}^{r-1}
		\mathbf1_{\Omega_{\Lambda,j}}
		\inpro{
			\bff\theta_h^{j+1}-\bff\theta_h^j
		}{\bff\theta_h^{j+1}}_h
		+
		c_1k\sum_{j=0}^{r-1}
		\mathbf1_{\Omega_{\Lambda,j}}
		\norm{\partial_x\bff\theta_h^{j+1}}{\bb L^2}^2
		\Bigg]
		\nonumber\\
		&\qquad\ge
		c\max_{1\le r\le n}
		\mathbf1_{\Omega_{\Lambda,r-1}}
		\norm{\bff\theta_h^r}{h}^2
		+
		c\sum_{j=0}^{n-1}
		\mathbf1_{\Omega_{\Lambda,j}}
		\norm{\bff\theta_h^{j+1}-\bff\theta_h^j}{h}^2
		+
		ck\sum_{j=0}^{n-1}
		\mathbf1_{\Omega_{\Lambda,j}}
		\norm{\partial_x\bff\theta_h^{j+1}}{\bb L^2}^2.
		\label{eq:stopped-coercivity-with-gradient-sllg}
	\end{align}
	We apply \eqref{eq:stopped-coercivity-with-gradient-sllg} to
	\eqref{eq:error-summed-corrected-drift-sllg}, take the maximum over
	$r=1,\ldots,n$, and then take expectations. We also aim to apply \eqref{eq:stochastic-error-term-estimate-sllg} with
	$\varepsilon=1-2\gamma$ to estimate the last term in \eqref{eq:error-summed-corrected-drift-sllg}.
	Since $\bff\theta_h^0=\bff0$ and
	$\Omega_{\Lambda,j}\subset\Omega_{\Lambda,j-1}$ for $j\ge 1$, we have
	\[
	\sup_{0\le j\le n}
	\mathbf1_{\Omega_{\Lambda,j}}
	\norm{\bff\theta_h^j}{h}^2
	\le
	\max_{1\le r\le n}
	\mathbf1_{\Omega_{\Lambda,r-1}}
	\norm{\bff\theta_h^r}{h}^2.
	\]
	Therefore, the terms multiplied by $\delta$ and $\eta$ in
	\eqref{eq:stochastic-error-term-estimate-sllg} can be absorbed into,
	respectively, the discrete time-increment term and the stopped maximum
	on the left-hand side by choosing $\delta,\eta>0$ sufficiently small.
	We then infer that
	\begin{align}
		&\mathbb E\left[
		\max_{1\le r\le n}
		\mathbf1_{\Omega_{\Lambda,r-1}}
		\norm{\bff\theta_h^r}{h}^2
		\right]
		+
		\mathbb E\left[
		\sum_{j=0}^{n-1}
		\mathbf1_{\Omega_{\Lambda,j}}
		\norm{\bff\theta_h^{j+1}-\bff\theta_h^j}{h}^2
		\right]
		+
		\mathbb E\left[
		k\sum_{j=0}^{n-1}
		\mathbf1_{\Omega_{\Lambda,j}}
		\norm{\partial_x\bff\theta_h^{j+1}}{\bb L^2}^2
		\right]
		\nonumber\\
		&\qquad\le
		C_\Lambda k\sum_{j=0}^{n-1}
		\mathbb E\left[
		\max_{1\le r\le j+1}
		\mathbf1_{\Omega_{\Lambda,r-1}}
		\norm{\bff\theta_h^r}{h}^2
		\right]
		+
		C\mathbb E\left[
		\sum_{j=0}^{n-1}
		\frac1k\norm{\mathcal R_h^{j+1}}{-1,h}^2
		\right]
		\nonumber\\
		&\qquad\quad+
		C_\Lambda\mathbb E\left[
		k\sum_{j=0}^{n-1}
		\sum_{\nu=1}^{3}
		\left(
		\norm{I_h(\bff m^{j+1}-\bff m^j)}{1,h}^{2\nu}
		+\norm{\bff m_h^{j+1}-\bff m_h^j}{1,h}^{2\nu}
		\right)
		\right]
		+C_{\Lambda,\gamma}k^{2\gamma}
		\nonumber\\
		&\qquad
		\leq
		C_\Lambda k\sum_{j=0}^{n-1}
		\mathbb E\left[
		\max_{1\le r\le j+1}
		\mathbf1_{\Omega_{\Lambda,r-1}}
		\norm{\bff\theta_h^r}{h}^2
		\right]
		+C_{\Lambda,\gamma}\left(h^2+k^{2\gamma}\right),
		\label{eq:local-error-gronwall-ready-sllg}
	\end{align}
	where in the last step we used Lemma~\ref{lem:consistency-sllg} and
	\eqref{eq:Xi-increment-defect-sum-sllg}.
	Now, choose $k_\Lambda>0$ so that $C_\Lambda k_\Lambda\le\frac12$.
	The term with $j=n-1$ can then be absorbed into the left-hand side.
	The discrete Gronwall inequality then yields~\eqref{eq:theta-error-local-sllg}.
	Finally, \eqref{eq:true-error-local-sllg} follows as a consequence of \eqref{eq:error-splitting-sllg}, \eqref{eq:rho-interpolation-error-sllg}, and \eqref{eq:theta-error-local-sllg}.
\end{proof}

\begin{corollary}[Convergence rate in probability]
	\label{cor:convergence-rate-probability-sllg}
	Under the assumptions of Theorem~\ref{thm:localised-error-sllg},
	let $\gamma\in(0,\frac12)$ and define
	\[
		\mathfrak E_{h,k}
		:=
		\max_{0\le j\le J}
		\norm{\bff m^j-\bff m_h^j}{\bb L^2}^2
		+
		k\sum_{j=0}^{J-1}
		\norm{
			\partial_x(\bff m^{j+1}-\bff m_h^{j+1})
		}{\bb L^2}^2.
	\]
	Then
	\begin{equation}\label{eq:convergence-rate-probability-sllg}
		\lim_{M\to\infty}
		\limsup_{h,k\to0}\,
		\mathbb P\left[
		\mathfrak E_{h,k}>
		M\left(h^2+k^{2\gamma}\right)
		\right]
		=0.
	\end{equation}
\end{corollary}

\begin{proof}
	Fix $\Lambda>0$. Since $k_\Lambda>0$ is fixed once $\Lambda$ is
	fixed, one has $k\le k_\Lambda$ for all sufficiently small $k$.
	Hence, Theorem~\ref{thm:localised-error-sllg}, estimate
	\eqref{eq:localisation-probability-sllg}, and Markov's inequality give,
	for every $M>0$,
	\[
		\mathbb P\left[
		\mathfrak E_{h,k}>
		M\left(h^2+k^{2\gamma}\right)
		\right]
		\le
		\mathbb P[\Omega_{\Lambda,J}^{\complement}]
		+
		\mathbb P\left[
		\mathbf1_{\Omega_{\Lambda,J}}\mathfrak E_{h,k}
		>
		M\left(h^2+k^{2\gamma}\right)
		\right]
		\le
		\frac{C}{\Lambda}
		+
		\frac{C_{\Lambda,\gamma}}{M}.
	\]
	Therefore,
	\[
	\lim_{M\to\infty}
	\limsup_{h,k\to0}
	\mathbb P\left[
	\mathfrak E_{h,k}>
	M\left(h^2+k^{2\gamma}\right)
	\right]
	\le
	\frac{C}{\Lambda}.
	\]
	Since $\Lambda>0$ is arbitrary, letting $\Lambda\to\infty$ proves
	\eqref{eq:convergence-rate-probability-sllg}.
\end{proof}


\appendix

\section{Auxiliary inequalities}

\begin{lemma}[Discrete Jensen inequality]\label{lem:weighted-discrete-holder}
	Let $q\ge1$ and $k>0$. Let $\{a_j\}_{j=0}^{J-1}$ be a non-negative sequence and set $T=Jk$. Then
	\begin{equation}\label{eq:weighted-discrete-holder}
		\left(
		\sum_{j=0}^{J-1} k a_j
		\right)^q
		\le
		T^{q-1}
		\sum_{j=0}^{J-1} k a_j^q.
	\end{equation}
	The sign in \eqref{eq:weighted-discrete-holder} is reversed if $q\in (0,1)$.
\end{lemma}

\begin{lemma}[Discrete BDG inequality]\label{lem:discrete-BDG}
	Let $p\ge2$, let $H$ be a Hilbert space, and let
	$\{\bff Z_j\}_{j=0}^{J-1}$ be an $H$-valued martingale difference sequence
	with respect to a filtration $\{\mathcal F_j\}_{j=0}^{J}$, that is,
	$\bff Z_j$ is $\mathcal F_{j+1}$-measurable and
	\[
	\mathbb E[\bff Z_j\mid\mathcal F_j]=\bff0,
	\qquad
	j=0,\ldots,J-1.
	\]
	Assume that $\bff Z_j\in L^p(\Omega;H)$ for each $j$. Then there exists
	a constant $C_p>0$, depending only on $p$, such that
	\begin{equation}\label{eq:discrete-BDG-maximal}
		\left(
		\mathbb E
		\left[
		\max_{0\le r\le J}
		\norm{\sum_{j=0}^{r-1}\bff Z_j}{H}^p
		\right]
		\right)^{1/p}
		\le
		C_p
		\left(
		\sum_{j=0}^{J-1}
		\left(
		\mathbb E
		\left[
		\norm{\bff Z_j}{H}^p
		\right]
		\right)^{2/p}
		\right)^{\frac12}.
	\end{equation}
\end{lemma}

\begin{lemma}\label{lem:stopped-time-identity-sllg}
	Let $\{\Omega_j\}_{j=0}^J$ be a sequence of decreasing events.
	For any sequence $\{\bff v^j\}_{j=0}^J\subset\bb V_h$ and any
	$1\le n\le J$,
	\begin{align*}
		&\max_{1\le r\le n}
		\sum_{j=0}^{r-1}
		\mathbf1_{\Omega_j}
		\inpro{
			\bff v^{j+1}-\bff v^j
		}{
			\bff v^{j+1}
		}_h
		\ge
		\frac14
		\max_{1\le r\le n}
		\mathbf1_{\Omega_{r-1}}\norm{\bff v^r}{h}^2
		+
		\frac14
		\sum_{j=0}^{n-1}
		\mathbf1_{\Omega_j}
		\norm{\bff v^{j+1}-\bff v^j}{h}^2
		-
		\frac12
		\mathbf1_{\Omega_0}\norm{\bff v^0}{h}^2.
	\end{align*}
\end{lemma}

\begin{proof}
	For each $1\le r\le n$, the identity
	\[
	2\inpro{\bff v^{j+1}-\bff v^j}{\bff v^{j+1}}_h
	=
	\norm{\bff v^{j+1}}{h}^2
	-
	\norm{\bff v^j}{h}^2
	+
	\norm{\bff v^{j+1}-\bff v^j}{h}^2
	\]
	and summation by parts give
	\begin{align*}
		2\sum_{j=0}^{r-1}
		\mathbf1_{\Omega_j}
		\inpro{\bff v^{j+1}-\bff v^j}{\bff v^{j+1}}_h
		&=
		\mathbf1_{\Omega_{r-1}}\norm{\bff v^r}{h}^2
		-
		\mathbf1_{\Omega_0}\norm{\bff v^0}{h}^2
		\\
		&\qquad+
		\sum_{j=1}^{r-1}
		\left(
		\mathbf1_{\Omega_{j-1}}-\mathbf1_{\Omega_j}
		\right)
		\norm{\bff v^j}{h}^2
		+
		\sum_{j=0}^{r-1}
		\mathbf1_{\Omega_j}
		\norm{\bff v^{j+1}-\bff v^j}{h}^2.
	\end{align*}
	Since the events are decreasing, the penultimate term on the right is
	nonnegative. Hence,
	\begin{align*}
		&\sum_{j=0}^{r-1}
		\mathbf1_{\Omega_j}
		\inpro{\bff v^{j+1}-\bff v^j}{\bff v^{j+1}}_h
		+
		\frac12\mathbf1_{\Omega_0}\norm{\bff v^0}{h}^2
		\ge
		\frac12\mathbf1_{\Omega_{r-1}}\norm{\bff v^r}{h}^2
		+
		\frac12\sum_{j=0}^{r-1}
		\mathbf1_{\Omega_j}
		\norm{\bff v^{j+1}-\bff v^j}{h}^2.
	\end{align*}
	Taking the maximum over $r=1,\ldots,n$ proves the claim.
\end{proof}

\section{Technical consistency estimates}

\begin{lemma}\label{lem:full-H1-increment-moments-sllg}
	Assume the hypotheses of
	Proposition~\ref{prop:p-moment-stability-sllg} hold. Then, for every integer
	$p\ge1$,
	\begin{equation}\label{eq:full-H1-increment-moments-sllg}
		\mathbb E\left[
		\left(
		\sum_{j=0}^{J-1}
		\norm{\bff m_h^{j+1}-\bff m_h^j}{1,h}^2
		\right)^p
		\right]
		\le C_p.
	\end{equation}
\end{lemma}

\begin{proof}
	The nodewise estimate \eqref{eq:preliminary-L2-increment-sllg} gives
	\begin{align*}
		\sum_{j=0}^{J-1}
		\norm{\bff m_h^{j+1}-\bff m_h^j}{h}^2
		&\le
		Ck\left(
		k\sum_{j=0}^{J-1}
		\norm{
			\bff m_h^{j+\frac12}\times\Delta_h\bff m_h^{j+1}
		}{h}^2
		\right)
		+C\sum_{j=0}^{J-1}
		\norm{\hat{\Delta}_jW_h}{\bb L^\infty}^2.
	\end{align*}
	The first term on the right has moments of every order by
	Proposition~\ref{prop:p-moment-stability-sllg}. Moreover, the discrete Jensen inequality and~\eqref{eq:noise-increment-moment-sllg} imply that
	\[
	\begin{aligned}
		\mathbb E\left[
		\left(
		\sum_{j=0}^{J-1}
		\norm{\hat{\Delta}_jW_h}{\bb L^\infty}^2
		\right)^p
		\right]
		&\le
		J^{p-1}\sum_{j=0}^{J-1}
		\mathbb E\left[
		\norm{\hat{\Delta}_jW_h}{\bb L^\infty}^{2p}
		\right]
		\le C_p.
	\end{aligned}
	\]
	All moments of the gradient-increment contribution is controlled by the second term in
	\eqref{eq:p-moment-stability-sllg}. 
	These imply \eqref{eq:full-H1-increment-moments-sllg} for
	$p\ge2$. The case $p=1$ follows from the case $p=2$ and Jensen's
	inequality.
\end{proof}

In the following, we will establish several technical lemmas controlling the midpoint defect in the drift term $\mathscr A_h$.

\begin{lemma}\label{lem:endpoint-drift-monotonicity-sllg}
	Let $\mathscr A_h$ be the form given by \eqref{eq:discrete-drift-form-sllg} and $\Omega_{\Lambda,j}$ be defined by~\eqref{eq:localisation-set-sllg}.
	For every $\Lambda>0$ and $\eta>0$, there exists
	$C_{\Lambda,\eta}>0$, independent of $h$, $k$, and $j$, such that
	\begin{align}
		&\mathbf1_{\Omega_{\Lambda,j}}
		\Bigl[
		\mathscr A_h(
		I_h\bff m^{j+1};
		I_h\bff m^{j+1},\bff\theta_h^{j+1}
		)
		-
		\mathscr A_h(
		\bff m_h^{j+1};
		\bff m_h^{j+1},\bff\theta_h^{j+1}
		)
		\Bigr]
		\nonumber\\
		&\qquad\ge
		(\alpha\kappa-\eta)
		\mathbf1_{\Omega_{\Lambda,j}}
		\norm{\partial_x\bff\theta_h^{j+1}}{\bb L^2}^2
		-
		C_{\Lambda,\eta}
		\mathbf1_{\Omega_{\Lambda,j}}
		\norm{\bff\theta_h^{j+1}}{h}^2
		\nonumber\\
		&\qquad\quad
		-
		C_{\Lambda,\eta}
		\left(
		\norm{I_h(\bff m^{j+1}-\bff m^j)}{1,h}^4
		+
		\norm{\bff m_h^{j+1}-\bff m_h^j}{1,h}^4
		\right).
		\label{eq:endpoint-drift-monotonicity-sllg}
	\end{align}
\end{lemma}

\begin{proof}
	Note that both $I_h\bff m^{j+1}$ and $\bff m_h^{j+1}$ have unit length at
	every node. Hence, we have the identities
	\[
	\bff\theta_h^{j+1}(\bff{z}) \cdot I_h\bff m^{j+1}(\bff{z})
	=
	\frac12|\bff\theta_h^{j+1}(\bff{z})|^2,
	\qquad
	\bff\theta_h^{j+1}(\bff{z}) \cdot\bff m_h^{j+1}(\bff{z})
	=
	-\frac12|\bff\theta_h^{j+1}(\bff{z})|^2, \qquad
	\forall \bff{z}\in \mathcal{N}_h.
	\]
	Using the triple product identity
	$\bff a\times(\bff a\times\bff z)
	=(\bff a\cdot\bff z)\bff a-\bff z$ for $|\bff a|=1$,
	we obtain
	\begin{align}
		&\inpro{
			I_h\bff m^{j+1}
			\times
			\left(
			I_h\bff m^{j+1}
			\times\Delta_hI_h\bff m^{j+1}
			\right)
		}{\bff\theta_h^{j+1}}_h
		-
		\inpro{
			\bff m_h^{j+1}
			\times
			\left(
			\bff m_h^{j+1}
			\times\Delta_h\bff m_h^{j+1}
			\right)
		}{\bff\theta_h^{j+1}}_h
		\nonumber\\
		&\qquad=
		\norm{\partial_x\bff\theta_h^{j+1}}{\bb L^2}^2
		+
		\frac12
		\inpro{
			\Delta_hI_h\bff m^{j+1}
		}{
			I_h\left(
			|\bff\theta_h^{j+1}|^2I_h\bff m^{j+1}
			\right)
		}_h
		+
		\frac12
		\inpro{
			\Delta_h\bff m_h^{j+1}
		}{
			I_h\left(
			|\bff\theta_h^{j+1}|^2\bff m_h^{j+1}
			\right)
		}_h.
		\label{eq:endpoint-damping-identity-sllg}
	\end{align}
	The elementwise product estimate \eqref{eq:Ih-product-H1-sllg} and the nodal unit-length bound \eqref{eq:nodal-constraint-sllg} give
	\begin{align*}
		&
		\norm{
			\partial_x
			I_h\left(
			|\bff\theta_h^{j+1}|^2I_h\bff m^{j+1}
			\right)
		}{\bb L^2}
		\le
		C
		\norm{\bff\theta_h^{j+1}}{\bb L^\infty}
		\norm{\partial_x\bff\theta_h^{j+1}}{\bb L^2}
		+
		C
		\norm{\bff\theta_h^{j+1}}{\bb L^\infty}^2
		\norm{\partial_xI_h\bff m^{j+1}}{\bb L^2},
	\end{align*}
	and analogously with $I_h\bff m^{j+1}$ replaced by
	$\bff m_h^{j+1}$. Therefore, using the definition of $\Delta_h$ in \eqref{eq:def-discrete-laplacian-sllg} to discretely integrate by parts,
	\begin{align}
		&\left|
		\frac12
		\inpro{
			\Delta_hI_h\bff m^{j+1}
		}{
			I_h\left(
			|\bff\theta_h^{j+1}|^2I_h\bff m^{j+1}
			\right)
		}_h
		+
		\frac12
		\inpro{
			\Delta_h\bff m_h^{j+1}
		}{
			I_h\left(
			|\bff\theta_h^{j+1}|^2\bff m_h^{j+1}
			\right)
		}_h
		\right|
		\nonumber\\
		&\qquad\le
		C
		\left(
		\norm{\partial_xI_h\bff m^{j+1}}{\bb L^2}
		+
		\norm{\partial_x\bff m_h^{j+1}}{\bb L^2}
		\right)
		\norm{\bff\theta_h^{j+1}}{\bb L^\infty}
		\norm{\partial_x\bff\theta_h^{j+1}}{\bb L^2}
		\nonumber\\
		&\qquad\quad+
		C
		\left(
		\norm{\partial_xI_h\bff m^{j+1}}{\bb L^2}^2
		+
		\norm{\partial_x\bff m_h^{j+1}}{\bb L^2}^2
		\right)
		\norm{\bff\theta_h^{j+1}}{\bb L^\infty}^2
		\nonumber\\
		&\qquad
		\le
		\eta
		\norm{\partial_x\bff\theta_h^{j+1}}{\bb L^2}^2
		+
		C_\eta
		\left(
		1+
		\norm{\partial_xI_h\bff m^{j+1}}{\bb L^2}^4
		+
		\norm{\partial_x\bff m_h^{j+1}}{\bb L^2}^4
		\right)
		\norm{\bff\theta_h^{j+1}}{h}^2,
		\label{eq:endpoint-damping-remainder-first-sllg}
	\end{align}
	where in the last step we used the discrete Agmon inequality~\eqref{eq:discrete-agmon-sllg} and Young's inequality.
	
	For the precession part, we write
	\[
	\begin{aligned}
		&I_h\bff m^{j+1}\times\Delta_hI_h\bff m^{j+1}
		-
		\bff m_h^{j+1}\times\Delta_h\bff m_h^{j+1}
		=
		I_h\bff m^{j+1}\times\Delta_h\bff\theta_h^{j+1}
		+
		\bff\theta_h^{j+1}\times\Delta_h\bff m_h^{j+1}.
	\end{aligned}
	\]
	The second term is orthogonal to $\bff\theta_h^{j+1}$ at every node.
	Hence, we obtain by \eqref{eq:Ih-product-H1-sllg},
	\begin{align}
		&\left|
		\inpro{
			I_h\bff m^{j+1}\times\Delta_hI_h\bff m^{j+1}
			-
			\bff m_h^{j+1}\times\Delta_h\bff m_h^{j+1}
		}{
			\bff\theta_h^{j+1}
		}_h
		\right|
		\nonumber\\
		&\qquad
		=
		\left|
		\inpro{
			\partial_x\bff\theta_h^{j+1}
		}{
			\partial_xI_h\left(
			\bff\theta_h^{j+1}\times I_h\bff m^{j+1}
			\right)
		}
		\right|
		\nonumber\\
		&\qquad
		\le
		\norm{\partial_xI_h\bff m^{j+1}}{\bb L^2}
		\norm{\bff\theta_h^{j+1}}{\bb L^\infty}
		\norm{\partial_x\bff\theta_h^{j+1}}{\bb L^2}
		\nonumber\\
		&\qquad
		\leq
		\eta\norm{\partial_x\bff\theta_h^{j+1}}{\bb L^2}^2
		+
		C_\eta
		\left(
		1+
		\norm{\partial_xI_h\bff m^{j+1}}{\bb L^2}^4
		\right)
		\norm{\bff\theta_h^{j+1}}{h}^2,
		\label{eq:endpoint-precession-remainder-sllg}
	\end{align}
	where in the last step we again applied the discrete Agmon and Young inequalities.	
	Now, on $\Omega_{\Lambda,j}$,
	\[
	\begin{aligned}
		\norm{\partial_xI_h\bff m^{j+1}}{\bb L^2}^4
		&\le
		C_\Lambda
		+
		C\norm{I_h(\bff m^{j+1}-\bff m^j)}{1,h}^4,
		\\
		\norm{\partial_x\bff m_h^{j+1}}{\bb L^2}^4
		&\le
		C_\Lambda
		+
		C\norm{\bff m_h^{j+1}-\bff m_h^j}{1,h}^4.
	\end{aligned}
	\]
	Moreover, the nodal unit-length property implies $\norm{\bff\theta_h^{j+1}}{h}^2\le 4|D|$.
	Combining these estimates with
	\eqref{eq:endpoint-damping-identity-sllg}--
	\eqref{eq:endpoint-precession-remainder-sllg}, and relabelling
	$\eta$, proves \eqref{eq:endpoint-drift-monotonicity-sllg}.
\end{proof}

\begin{lemma}\label{lem:midpoint-endpoint-perturbation-sllg}
	Let $\mathscr A_h$ be the form given by \eqref{eq:discrete-drift-form-sllg} and $\Omega_{\Lambda,j}$ be defined by~\eqref{eq:localisation-set-sllg}.
	For every $\Lambda>0$ and $\eta>0$, there exists
	$C_{\Lambda,\eta}>0$, independent of $h$, $k$, and $j$, such that
	\begin{align}
		&\mathbf1_{\Omega_{\Lambda,j}}
		\Bigg|
		\mathscr A_h\left(
		I_h\overline{\bff m}^{j+\frac12};
		I_h\bff m^{j+1},\bff\theta_h^{j+1}
		\right)
		-
		\mathscr A_h\left(
		I_h\bff m^{j+1};
		I_h\bff m^{j+1},\bff\theta_h^{j+1}
		\right)
		\Bigg|
		\nonumber\\
		&\qquad\le
		\eta\mathbf1_{\Omega_{\Lambda,j}}
		\norm{\partial_x\bff\theta_h^{j+1}}{\bb L^2}^2
		+
		C_{\Lambda,\eta}
		\mathbf1_{\Omega_{\Lambda,j}}
		\norm{\bff\theta_h^{j+1}}{h}^2
		+
		C_{\Lambda,\eta}
		\sum_{\nu=1}^{3}
		\norm{I_h(\bff m^{j+1}-\bff m^j)}{1,h}^{2\nu},
		\label{eq:exact-midpoint-perturbation-local-sllg}
	\end{align}
	and
	\begin{align}
		&\mathbf1_{\Omega_{\Lambda,j}}
		\Bigg|
		\mathscr A_h\left(
		\bff m_h^{j+\frac12};
		\bff m_h^{j+1},\bff\theta_h^{j+1}
		\right)
		-
		\mathscr A_h\left(
		\bff m_h^{j+1};
		\bff m_h^{j+1},\bff\theta_h^{j+1}
		\right)
		\Bigg|
		\nonumber\\
		&\qquad\le
		\eta\mathbf1_{\Omega_{\Lambda,j}}
		\norm{\partial_x\bff\theta_h^{j+1}}{\bb L^2}^2
		+
		C_{\Lambda,\eta}
		\mathbf1_{\Omega_{\Lambda,j}}
		\norm{\bff\theta_h^{j+1}}{h}^2
		+
		C_{\Lambda,\eta}
		\sum_{\nu=1}^{3}
		\norm{\bff m_h^{j+1}-\bff m_h^j}{1,h}^{2\nu}.
		\label{eq:numerical-midpoint-perturbation-local-sllg}
	\end{align}
\end{lemma}

\begin{proof}
	We prove \eqref{eq:exact-midpoint-perturbation-local-sllg}, while noting that \eqref{eq:numerical-midpoint-perturbation-local-sllg} follows by the same argument.
	
	By noting that $I_h\overline{\bff m}^{j+\frac12}
	=
	I_h\bff m^{j+1}
	-\frac12I_h(\bff m^{j+1}-\bff m^j)$, we can then write
	\begin{align}
		&\mathscr A_h\left(
		I_h\overline{\bff m}^{j+\frac12};
		I_h\bff m^{j+1},\bff\theta_h^{j+1}
		\right)
		-
		\mathscr A_h\left(
		I_h\bff m^{j+1};
		I_h\bff m^{j+1},\bff\theta_h^{j+1}
		\right)
		\nonumber\\
		&\quad=
		\frac{\kappa}{2}
		\inpro{
			I_h(\bff m^{j+1}-\bff m^j)
			\times\Delta_hI_h\bff m^{j+1}
		}{
			\bff\theta_h^{j+1}
		}_h
		\nonumber\\
		&\qquad-
		\frac{\alpha\kappa}{2}
		\inpro{
			I_h(\bff m^{j+1}-\bff m^j)
			\times
			\left(
			I_h\bff m^{j+1}
			\times\Delta_hI_h\bff m^{j+1}
			\right)
		}{
			\bff\theta_h^{j+1}
		}_h
		\nonumber\\
		&\qquad-
		\frac{\alpha\kappa}{2}
		\inpro{
			I_h\bff m^{j+1}
			\times
			\left(
			I_h(\bff m^{j+1}-\bff m^j)
			\times\Delta_hI_h\bff m^{j+1}
			\right)
		}{
			\bff\theta_h^{j+1}
		}_h
		\nonumber\\
		&\qquad+
		\frac{\alpha\kappa}{4}
		\inpro{
			I_h(\bff m^{j+1}-\bff m^j)
			\times
			\left(
			I_h(\bff m^{j+1}-\bff m^j)
			\times\Delta_hI_h\bff m^{j+1}
			\right)
		}{
			\bff\theta_h^{j+1}
		}_h.
		\label{eq:exact-midpoint-perturbation-expansion-sllg}
	\end{align}
	Using the definition of $\Delta_h$ in
	\eqref{eq:def-discrete-laplacian-sllg} to integrate by parts discretely, together with the elementwise product
	estimate \eqref{eq:Ih-product-H1-sllg} and the nodal bound \eqref{eq:nodal-constraint-sllg}, we obtain
	\begin{align}
		&\left|
		\mathscr A_h\left(
		I_h\overline{\bff m}^{j+\frac12};
		I_h\bff m^{j+1},\bff\theta_h^{j+1}
		\right)
		-
		\mathscr A_h\left(
		I_h\bff m^{j+1};
		I_h\bff m^{j+1},\bff\theta_h^{j+1}
		\right)
		\right|
		\nonumber\\
		&\qquad\le
		C
		\norm{\partial_xI_h\bff m^{j+1}}{\bb L^2}
		\norm{I_h(\bff m^{j+1}-\bff m^j)}{\bb L^\infty}
		\norm{\partial_x\bff\theta_h^{j+1}}{\bb L^2}
		\nonumber\\
		&\qquad\quad+
		C
		\norm{\partial_xI_h\bff m^{j+1}}{\bb L^2}
		\norm{\bff\theta_h^{j+1}}{\bb L^\infty}
		\norm{\partial_xI_h(\bff m^{j+1}-\bff m^j)}{\bb L^2}
		\nonumber\\
		&\qquad\quad+
		C
		\norm{\partial_xI_h\bff m^{j+1}}{\bb L^2}^2
		\norm{\bff\theta_h^{j+1}}{\bb L^\infty}
		\norm{I_h(\bff m^{j+1}-\bff m^j)}{\bb L^\infty}.
		\label{eq:exact-midpoint-perturbation-before-young-sllg}
	\end{align}
	Now, on $\Omega_{\Lambda,j}$, we can estimate
	\[
	\norm{\partial_xI_h\bff m^{j+1}}{\bb L^2}
	\le
	C_\Lambda
	+
	\norm{I_h(\bff m^{j+1}-\bff m^j)}{1,h},
	\]
	while applying the discrete Agmon inequality to the terms with $\bb{L}^\infty$ norms.
	Young's inequality applied term by term in
	\eqref{eq:exact-midpoint-perturbation-before-young-sllg} then yields~\eqref{eq:exact-midpoint-perturbation-local-sllg}.
	
	To show~\eqref{eq:numerical-midpoint-perturbation-local-sllg}, we use
	$\bff m_h^{j+\frac12}= \bff m_h^{j+1} -\frac12(\bff m_h^{j+1}-\bff m_h^j)$,
	noting that
	\[
	\norm{\partial_x\bff m_h^{j+1}}{\bb L^2}
	\le
	C_\Lambda
	+
	\norm{\bff m_h^{j+1}-\bff m_h^j}{1,h}
	\]
	on $\Omega_{\Lambda,j}$, then repeat the same argument. This completes the proof.
\end{proof}

The above two lemmas are used to derive the following estimate.

\begin{lemma}\label{lem:local-drift-summable-defect-sllg}
	Let $\mathscr A_h$ be the form given by \eqref{eq:discrete-drift-form-sllg} and $\Omega_{\Lambda,j}$ be defined by~\eqref{eq:localisation-set-sllg}.
	For every fixed $\Lambda>0$, there exist $c_0>0$ and $C_\Lambda>0$,
	independent of $h$, $k$, and $j$, such that
	\begin{align}
		&k\mathbf1_{\Omega_{\Lambda,j}}
		\left[
		\mathscr A_h\left(
		I_h\overline{\bff m}^{j+\frac12};
		I_h\bff m^{j+1},\bff\theta_h^{j+1}
		\right)
		-
		\mathscr A_h\Bigl(
		\bff m_h^{j+\frac12};
		\bff m_h^{j+1},\bff\theta_h^{j+1}
		\Bigr)
		\right]
		\nonumber\\
		&\qquad\ge
		c_0k\mathbf1_{\Omega_{\Lambda,j}}
		\norm{\partial_x\bff\theta_h^{j+1}}{\bb L^2}^2
		-
		C_\Lambda k\mathbf1_{\Omega_{\Lambda,j}}
		\norm{\bff\theta_h^{j+1}}{h}^2
		\nonumber\\
		&\qquad\quad-
		C_\Lambda k
		\sum_{\nu=1}^{3}
		\left(
		\norm{I_h(\bff m^{j+1}-\bff m^j)}{1,h}^{2\nu}
		+
		\norm{\bff m_h^{j+1}-\bff m_h^j}{1,h}^{2\nu}
		\right).
		\label{eq:local-drift-summable-defect-sllg}
	\end{align}
	Moreover, for every $\gamma\in(0,\frac12)$, there exists $C_\gamma>0$ such that
	\begin{align}
		&\mathbb E\left[
		k\sum_{j=0}^{J-1}
		\sum_{\nu=1}^{3}
		\left(
		\norm{I_h(\bff m^{j+1}-\bff m^j)}{1,h}^{2\nu}
		+
		\norm{\bff m_h^{j+1}-\bff m_h^j}{1,h}^{2\nu}
		\right)
		\right]
		\le
		C_\gamma k^{2\gamma}.
		\label{eq:Xi-increment-defect-sum-sllg}
	\end{align}
\end{lemma}

\begin{proof}
	We add and subtract the two endpoint drift terms to write
	\begin{align*}
		&\mathscr A_h\left(
		I_h\overline{\bff m}^{j+\frac12};
		I_h\bff m^{j+1},\bff\theta_h^{j+1}
		\right)
		-
		\mathscr A_h\left(
		\bff m_h^{j+\frac12};
		\bff m_h^{j+1},\bff\theta_h^{j+1}
		\right)
		\\
		&\quad=
		\mathscr A_h\left(
		I_h\bff m^{j+1};
		I_h\bff m^{j+1},\bff\theta_h^{j+1}
		\right)
		-
		\mathscr A_h\left(
		\bff m_h^{j+1};
		\bff m_h^{j+1},\bff\theta_h^{j+1}
		\right)
		\\
		&\qquad+
		\Bigg[
		\mathscr A_h\left(
		I_h\overline{\bff m}^{j+\frac12};
		I_h\bff m^{j+1},\bff\theta_h^{j+1}
		\right)
		-
		\mathscr A_h\left(
		I_h\bff m^{j+1};
		I_h\bff m^{j+1},\bff\theta_h^{j+1}
		\right)
		\Bigg]
		\\
		&\qquad-
		\Bigg[
		\mathscr A_h\left(
		\bff m_h^{j+\frac12};
		\bff m_h^{j+1},\bff\theta_h^{j+1}
		\right)
		-
		\mathscr A_h\left(
		\bff m_h^{j+1};
		\bff m_h^{j+1},\bff\theta_h^{j+1}
		\right)
		\Bigg].
	\end{align*}
	We can now apply Lemma~\ref{lem:endpoint-drift-monotonicity-sllg} to the first term on the right-hand side and Lemma~\ref{lem:midpoint-endpoint-perturbation-sllg} to the
	two remaining differences. Choosing $\eta>0$ sufficiently small, for
	instance $3\eta\le\frac12\alpha\kappa$, gives
	\eqref{eq:local-drift-summable-defect-sllg} with
	$c_0=\frac12\alpha\kappa$ after enlarging $C_\Lambda$.
	
	It remains to prove~\eqref{eq:Xi-increment-defect-sum-sllg}. By the
	$\bb H^1$-stability of $I_h$ and
	\eqref{eq:exact-solution-regularity-sllg}, for $\nu=1,2,3$, we have
	\begin{align}
		\mathbb E\left[
		k\sum_{j=0}^{J-1}
		\norm{I_h(\bff m^{j+1}-\bff m^j)}{1,h}^{2\nu}
		\right]
		\le
		C
		\mathbb E\left[
		k\sum_{j=0}^{J-1}
		\norm{\bff m^{j+1}-\bff m^j}{\bb H^1}^{2\nu}
		\right]
		\le
		C_\gamma k^{2\nu\gamma}
		\le
		C_\gamma k^{2\gamma}.
		\label{eq:exact-increment-polynomial-sum-sllg}
	\end{align}
	Moreover, for the numerical increments, by Lemma~\ref{lem:full-H1-increment-moments-sllg} we have
	\begin{align}
		\mathbb E\left[
		k\sum_{j=0}^{J-1}
		\norm{\bff m_h^{j+1}-\bff m_h^j}{1,h}^{2\nu}
		\right]
		\le
		k\,
		\mathbb E\left[
		\left(
		\sum_{j=0}^{J-1}
		\norm{\bff m_h^{j+1}-\bff m_h^j}{1,h}^2
		\right)^\nu
		\right]
		\le
		C_\nu k.
		\label{eq:num-increment-polynomial-sum-sllg}
	\end{align}
	Since $0<k\le1$ and $2\gamma<1$, the required estimate~\eqref{eq:Xi-increment-defect-sum-sllg} follows from~\eqref{eq:exact-increment-polynomial-sum-sllg} and \eqref{eq:num-increment-polynomial-sum-sllg}.
\end{proof}



\newcommand{\noopsort}[1]{}\def\cprime{$'$}
\def\soft#1{\leavevmode\setbox0=\hbox{h}\dimen7=\ht0\advance \dimen7
	by-1ex\relax\if t#1\relax\rlap{\raise.6\dimen7
		\hbox{\kern.3ex\char'47}}#1\relax\else\if T#1\relax
	\rlap{\raise.5\dimen7\hbox{\kern1.3ex\char'47}}#1\relax \else\if
	d#1\relax\rlap{\raise.5\dimen7\hbox{\kern.9ex \char'47}}#1\relax\else\if
	D#1\relax\rlap{\raise.5\dimen7 \hbox{\kern1.4ex\char'47}}#1\relax\else\if
	l#1\relax \rlap{\raise.5\dimen7\hbox{\kern.4ex\char'47}}#1\relax \else\if
	L#1\relax\rlap{\raise.5\dimen7\hbox{\kern.7ex
			\char'47}}#1\relax\else\message{accent \string\soft \space #1 not
		defined!}#1\relax\fi\fi\fi\fi\fi\fi}

\end{document}